\documentclass[a4paper,10pt,reqno]{amsart}

\usepackage{amsmath}
\usepackage{amsthm}
\usepackage{amsfonts}
\usepackage{amssymb}
\usepackage{mathrsfs}
\usepackage[shortlabels]{enumitem}
\usepackage{graphicx}
\usepackage{subfigure}
\usepackage{color}
\usepackage[dvipsnames]{xcolor}

\usepackage{mathtools} 
\usepackage{courier} 
\usepackage[T1]{fontenc}  

\usepackage{hyperref}

\DeclarePairedDelimiter{\abs}{\lvert}{\rvert}
\DeclarePairedDelimiter{\norma}{\lVert}{\rVert}

\DeclareMathOperator{\Div}{div}

\newcommand{\numberset}{\mathbb}
\newcommand{\N}{\numberset{N}}

\newcommand{\Z}{\numberset{Z}}
\newcommand{\R}{\numberset{R}}
\newcommand{\diff}{\mathrm{d}}
\newcommand{\dom}{\mathrm{dom}}

\newcommand{\p}{\mathcal{P}}
\newcommand{\eps}{\epsilon}
\newcommand{\la}{\lambda}
\newcommand{\cH}{\mathcal{H}}
\newcommand{\ov}{\overline}

\newcommand{\Ccon}{\stackrel{\cC}{\rightarrow}}

\newcommand{\cE}{\,{\mathcal E}\,}
\newcommand{\cC}{\,{\mathcal C}\,}

\renewcommand\rho\varrho
\newcommand\al\alpha

\newcommand\ds\displaystyle

\newcommand{\dist}{\operatorname{dist}}

\newcommand{\supp}{\operatorname{supp}}

\newcommand{\beq}{\begin{equation}}
\newcommand{\eeq}{\end{equation}}
\newcommand{\be}{\begin{equation*}}
\newcommand{\ee}{\end{equation*}}
\newcommand{\bmat}{\begin{pmatrix}}
\newcommand{\emat}{\end{pmatrix}}

\newcommand{\Hmm}[1]{\leavevmode{\marginpar{\tiny%
$\hbox to 0mm{\hspace*{-0.5mm}$\leftarrow$\hss}%
\vcenter{\vrule depth 0.1mm height 0.1mm width \the\marginparwidth}%
\hbox to
0mm{\hss$\rightarrow$\hspace*{-0.5mm}}$\\\relax\raggedright #1}}}

\newtheorem{theorem}{Theorem}
\newtheorem{lemma}{Lemma}
\newtheorem{proposition}{Proposition}
\newtheorem{definition}{Definition}
\theoremstyle{plain}

\newtheorem{remark}{Remark}

\numberwithin{equation}{section}

\author[F.~Ferraresso]{Francesco Ferraresso}
\address{School of Mathematics, Cardiff University, Senghennydd Road, Cardiff, CF24 4AG, UK}
\email{FerraressoF@cardiff.ac.uk}

\date{\today}

\thanks{}
\usepackage[update,prepend]{epstopdf}

\title[]{On the spectral instability for weak intermediate triharmonic problems}

\begin{document}


\begin{abstract}{We define the weak intermediate boundary conditions for the triharmonic operator $- \Delta^3$. We analyse the sensitivity of
this type of boundary conditions upon domain perturbations. We construct a perturbation $(\Omega_\eps)_{\eps > 0}$ of a smooth domain $\Omega$ of $\R^N$ for which the weak intermediate boundary conditions on $\partial \Omega_\eps$ are not preserved in the limit on $\partial \Omega$, analogously to the Babu\v{s}ka paradox for the hinged plate. Four different boundary conditions can be produced in the limit, depending on the convergence of $\partial \Omega_\eps$ to $\partial \Omega$.  In one particular case, we obtain a ``strange'' boundary condition featuring a microscopic energy term related to the shape of the approaching domains.  Many aspects of our analysis could be generalised to an arbitrary order elliptic differential operator of order $2m$ and to more general domain perturbations.}
\end{abstract}

\maketitle

\section{Introduction}
Let $W$ be a smooth bounded domain of $\R^{N-1}$, $b \in C^{4}(W)$ be a periodic, positive function with period $Y = (-1/2, 1/2)^{N-1}$.
Let $\alpha \in (0,+ \infty)$ be fixed, and define
\begin{equation}\label{eq: geometric setting}
\begin{split}
\Omega_\eps &:= \bigg\{ x=(\bar{x}, x_N) \in \Omega : \bar{x} \in W, -1 < x_N < g_\eps(\bar{x}) = \eps^\alpha b\bigg(\frac{\bar{x}}{\eps}\bigg) \bigg\}\\
\Omega &:= W \times (-1,0),
\end{split}
\end{equation}
for $\eps \in (0,1]$. We consider the weak intermediate problem for the triharmonic operator $A_\eps = (- \Delta)^3 + I$ in $\Omega_\eps$, given by
\begin{equation} \label{triharmonic weak first}
\int_{\Omega_\eps} \big( D^3 u_\eps : D^3 \varphi + u_\eps \varphi \big) = \la(\Omega_\eps) \int_{\Omega_\eps} u_\eps \, \varphi, \quad \varphi \in H^3(\Omega_\eps) \cap H^1_0(\Omega_\eps),
\end{equation}
where $D^3 f : D^3 g = \sum_{i,j,k = 1, \dots N} \frac{\partial^3 f}{\partial x_i \partial x_j \partial x_k} \frac{\partial^3 g}{\partial x_i \partial x_j \partial x_k}$ is the Frobenius product of the two tensors $D^3f$ and $D^3 g$, $\la(\Omega_\eps)$ is the eigenvalue and $u_\eps \in H^3(\Omega_\eps) \cap H^1_0(\Omega_\eps)$ is the eigenfunction. Here and in the sequel $H^k$, $H^k_0$ denote the standard Sobolev spaces with regularity index $k$ and integrability index $2$.

We are interested in the behaviour of the solutions $u_\eps$ and of the eigenvalues $\la(\Omega_\eps)$ of \eqref{triharmonic weak first} as $\eps \to 0$. Note that $\Omega_\eps$ approaches $\Omega$ as $\eps \to 0$ in a rather singular way, since the function $g_\eps$ oscillates with very large frequency as $\eps \to 0$. It is worth noting that if $\alpha < 3$, it is not possible to construct a family of smooth diffeomorphisms $\Phi_\eps : \Omega \to \Omega_\eps$ such that $\norma{\Phi_\eps - I}_{C^3(\R^N, \R^N)} \to 0$ as $\eps \to 0$. Therefore classical and elegant techniques based on the direct comparison of the Rayleigh quotients associated to $\la(\Omega_\eps)$ and $\la(\Omega)$ do not work in general in the singular setting described in \eqref{eq: geometric setting}.

Polyharmonic operators $(-\Delta)^m$ with intermediate or Neumann boundary conditions are known to be rather sensitive to variation of the domains in $\R^N$, $N > 1$. See for example \cite{BuosoLamb} for regular perturbations, and \cite{ArrLamb, FerreroLamb, MR3762319} for more singular settings. When $m=1$, unexpected limiting behaviour of the eigenvalues of the Neumann Laplace operator $-\Delta_{\rm neu}$ is well-known since the `dumbbell' example in \cite{CouHil}, where $\la_2(\Omega_\eps) \to 0$ as $\eps \to 0$ instead of converging to
$\la_2(\Omega) > 0$. More in general, let $R_\eps$ be a smooth domain of $\R^N$ converging (in Hausdorff sense) to a lower dimensional set $D \subset \R^d$, $d < N$, and let $\Omega_\eps$ be the smooth domain obtained by attaching $R_\eps$ to a bounded smooth domain $\Omega \subset \R^N$. Then, the eigenvalues of $-\Delta_{\rm neu}$ in $\Omega_\eps$ will not converge only to the respective eigenvalues in $\Omega$, but also to the eigenvalues of a differential problem in $D$. Indeed, the eigenvalues of $-\Delta_{\rm neu}$ on $R_\eps$ are known to converge to the eigenvalues of the Laplace-Beltrami operator on $D$, see e.g., \cite{sch}.  See also \cite{HalRau} for more results for reaction-diffusion operators on thin domains, \cite{Jim, JimKos} for dumbbell-type domains,\cite{ArrVil1, ArrVil2} for domains with fast oscillating boundaries and \cite{CoDaDRMus, DRMus} for domains with small holes. When $m =2$ the Babu\v{s}ka Paradox for the biharmonic operator $\Delta^2_{\rm SBC}$ shows that intermediate boundary conditions are not stable under polygonal approximation of a smooth domain in $\R^2$, see \cite{MazNaz} and the introduction to \cite{FerLamb} for further details. Elliptic operators of order $2m$ with $m \geq 2$ and diverse boundary conditions have been recently considered in \cite{ColPro, BuoKen} and in the preprint \cite{FerPro} where it is shown that the eigenvalues of the biharmonic operator with Neumann boundary conditions on a thin domain converge, as the size of the domain tends to zero, to the eigenvalues of a system of equations on the boundary. The common thread in these examples is the lack of \textit{spectral stability}, see Def. \ref{spectral stability}; roughly speaking, a sequence of operators $(A_n)_n$ satisfying \emph{the same} boundary conditions is spectrally stable if it is spectrally exact (in the sense of \cite{MR1205756}) and the limiting operator $A$ satisfies the same boundary conditions as the operators $A_n$, $n \in \N$. \\
Problem \eqref{triharmonic weak first} is an interesting example of spectral instability, which, according to \cite{ArrLamb, FerLamb}, can be regarded as a smooth version of the Babu\v{s}ka paradox. In order to describe our main result, it is convenient to define a class of triharmonic problems with different boundary conditions. For every $\eps > 0$, let $V(\Omega_\eps)$ be a linear subspace of $H^3(\Omega_\eps)$ containing $H^3_0(\Omega_\eps)$. Assume that $V(\Omega_\eps)$ is compactly embedded in $L^2(\Omega_\eps)$, and it is complete with respect to the $H^3(\Omega_\eps)$ norm, which is induced by the quadratic form
\[
Q_{\Omega_\eps}(u) = \int_{\Omega_\eps} |D^3 u|^2 + |u|^2, \quad \quad u \in V(\Omega_\eps).
\]

\begin{figure}
\centering
\includegraphics[width=0.7\textwidth]{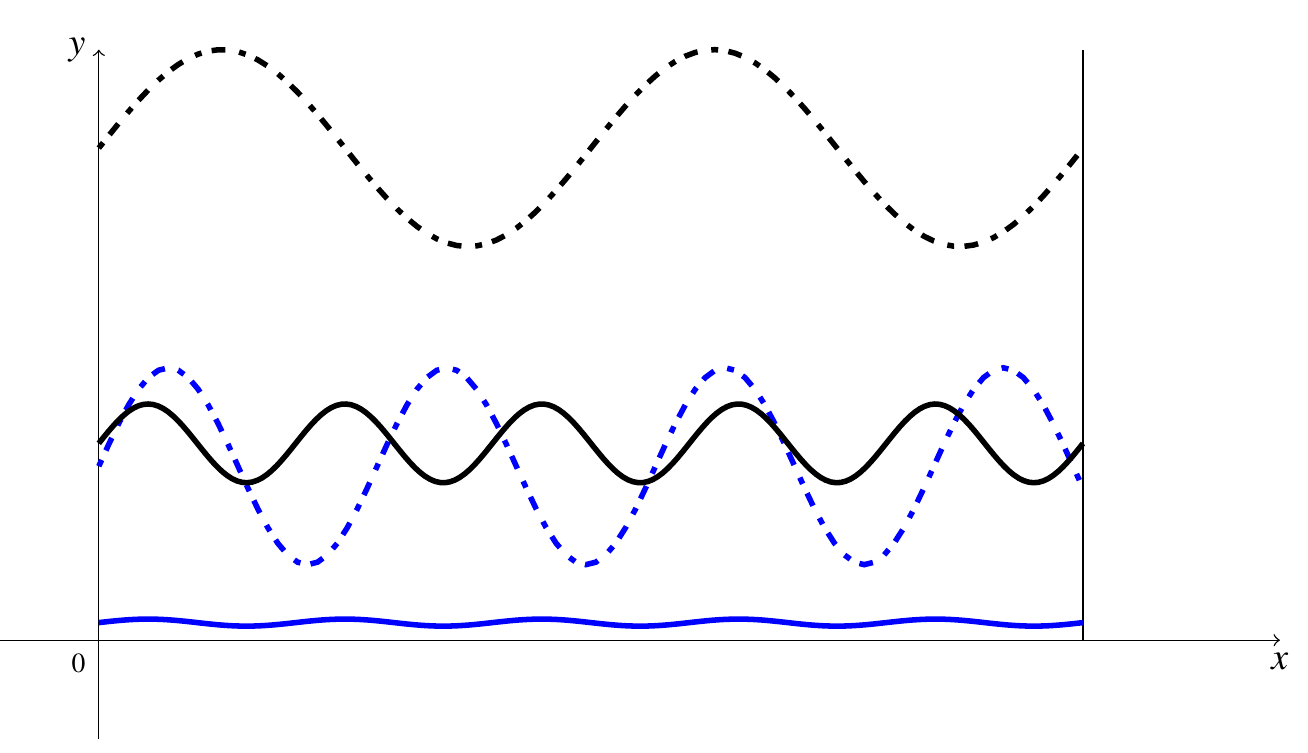}
\caption{Graph of $g_\eps(x) = \eps^\alpha b(x/\eps)$ with $b(y) = 10 + 2 \sin (\pi y/5)$. Black colour corresponds to $\alpha = 1$, blue to
$\alpha = 5/2$. The dashed line corresponds to $\epsilon = 0.5$, the thick line to $\epsilon = 0.2$. The blue graph flattens out much faster
than the black one as $\eps \to 0$.}
\label{fig: oscillating boundaries}
\end{figure}

We then define
\begin{equation} \label{triharmonic general}
\int_{\Omega_\eps} \big( D^3 u_\eps : D^3 \varphi + u_\eps \varphi \big) = \la(\Omega_\eps) \int_{\Omega_\eps} u_\eps \varphi, \quad \varphi \in V(\Omega_\eps).
\end{equation}
By the second representation theorem \cite[Theorem 2.23, VI.2]{Kato}, the sesquilinear form in \eqref{triharmonic general} is associated to a positive self-adjoint operator $A_{V(\Omega_\eps)} := (-\Delta)_{V(\Omega_\eps)}^3 + I$. The inverse $A_{V(\Omega_\eps)}^{-1}$ is a compact operator in $L^2(\Omega_\eps)$, due to the compact embedding of $V(\Omega_\eps)$ in $L^2(\Omega_\eps)$. Thus, the spectrum associated to \eqref{triharmonic general} is discrete and consists of an unbounded sequence of positive eigenvalues of finite multiplicity $(\la_j(\Omega_\eps))_{j \in \N}$.

It is shown in \cite{FerLamb} that if $V(\Omega_\eps) = H^3(\Omega_\eps) \cap H^2_0(\Omega_\eps)$ for all $\eps \in (0,1]$ then the spectrum $(\la_j(\Omega_\eps))_{j \in \N}$ of problem \eqref{triharmonic general} approaches the spectrum $(\la_k(\Omega))_{k \in \N}$ of the same problem \eqref{triharmonic general} with $\Omega$ in place of $\Omega_\eps$, provided that $\alpha > \frac{3}{2}$. If instead $0 < \alpha < \frac{3}{2}$ the limiting problem satisfies Dirichlet boundary conditions (corresponding to $V(\Omega) = H^3_0(\Omega)$ in \eqref{triharmonic general}) on $W \times \{0\}$. It is also shown that this Babu\v{s}ka-type paradox is shared by all the polyharmonic operators $(-\Delta)_{\rm SBC}^m$ with strong intermediate boundary conditions, shortened SBC, for which $V(\Omega_\eps) = H^m(\Omega_\eps) \cap H^{m-1}_0(\Omega_\eps)$ in the polyharmonic analogous of \eqref{triharmonic general}. In other words, polyharmonic operators with SBC are spectrally stable on $(\Omega_\eps)_{\eps \in [0,1]}$ provided that $\alpha > 3/2$.\\
As already pointed out in \cite{FerLamb}, there is another possible choice of intermediate boundary conditions for the triharmonic operator, the weak intermediate boundary conditions (shortened WBC) defined implicitly by \eqref{triharmonic weak first}. From the spectral stability result \cite[Theorem 4]{FerLamb} we know that if $\Omega_\eps$ and $\Omega$ are as in \eqref{eq: geometric setting} the sequence of operators $A_\eps = (- \Delta_{\rm WBC}^3 + I))_{\eps \in [0,1]}$ associated to \eqref{triharmonic weak first} is spectrally stable provided that $\alpha > 5/2$.\\
The main result of this article, see Thm.\ref{thm: spectral conv tri weak}, is the analysis of the case $\alpha \leq 5/2$. We prove that there are three different cases depending on $\alpha$, that can be summarised as follows
\begin{enumerate} [label =(\roman*)]
\item if $\alpha \in \big(\frac32, \frac52 \big)$ the eigenvalues $\la_j(\Omega_\eps)$ of \eqref{triharmonic weak first} converge in the limit to the eigenvalues of $(- \Delta)^3 + I$ with mixed WBC-SBC.
\item if $\alpha \in (0, 1)$, the limiting operator $(- \Delta)^3 + I$ satisfies mixed boundary conditions of type WBC-Dirichlet.
\item if $\alpha = \frac{5}{2}$ the limiting boundary value problem features a 'strange' boundary conditions which keeps track of the shape of the periodic function $b$ in \eqref{eq: geometric setting}.
\end{enumerate}
The case $\alpha \in (1, 3/2]$ is not considered in this article and it is left as an open problem, see Remark \ref{rem: open} for further explanations of why this range of values does not seem treatable with our method of proof.\\
While Theorem \ref{thm: spectral conv tri weak} look similar to \cite[Theorem 7]{FerLamb}, we point out that we had to face several new technical difficulties due to the extreme singularity of the perturbation $\Omega \mapsto \Omega_\eps$ when $\alpha \leq 5/2$. Indeed, the diffeomorphism $\Phi_\eps : \Omega \mapsto \Omega_\eps$ that we use in the proof of the main theorem has derivatives with strongly divergent $L^2$-norms as $\eps \to 0$. Furthermore, it is not possible to balance this unboundedness of the derivatives as in \cite{FerLamb}, where it was pivotal to exploit the vanishing of \textit{both} $u_\eps$ and $\frac{\partial u}{\partial n}$ at the boundary. Finally, the proof of \cite[Theorem 7]{FerLamb} in the degenerate case $\alpha \leq 3/2$ relies on \cite[Lemma 4.3]{CasDiaz}, for which it is fundamental that the critical threshold for the spectral stability is $\alpha = 3/2$. This condition is clearly not satisfied by weak intermediate problems. \\
To overcome these additional hurdles, we prove a new, yet rather technical degeneration result, see Lemma \ref{lemma: degeneration H^3 H^1_0}. Its proof involves a careful analysis of the behaviour of the derivatives of functions $u_\eps \in H^3(\Omega_\eps) \cap
H^1_0(\Omega_\eps)$ close to the oscillating boundaries. Broadly speaking, we need a combination of three arguments: (i) the use of the
anisotropic unfolding operator to control the $L^2$-norm of the derivatives of $u_\eps$ close to the oscillating boundary; (ii) the weighted
convergence of the traces of the unfolded functions $\hat{u}_\eps$ to the trace of the weak limit $u$ of the original functions $u_\eps$;
(iii) the use of the standard unfolding operators (which is equivalent to the so-called two-scale convergence) to deduce additional
information on the trace of $u$ when $\alpha \leq 1$ and $1 < \alpha < 2$.\\
We refer the reader to \cite{CioDo} for more details about homogenisation techniques and to \cite{CioDamGri} for the unfolding operator. The
use of the anisotropic unfolding operator and some of the techniques used in the proof of Lemma \ref{lemma: degeneration H^3 H^1_0} were
inspired by a careful reading of \cite{CasDiaz, CasDiazLuLaySuGrau} and by some classical asymptotic analysis techniques in the spirit of
\cite{MazNazPlaI, MazNazPlaII}. \\
This article is organised in the following way. In Section \ref{sec: main} we introduce the weak intermediate boundary conditions for the
triharmonic operator $-\Delta^3$, and we state the main result of the paper, Theorem \ref{thm: spectral conv tri weak}. In Section \ref{sec: auxil} we collect some standard results about the unfolding operator and the tangential calculus. In Section \ref{sec: spec stab} we recall some definitions and results about the convergence of bounded operators on varying Hilbert spaces, and we give the definitions of spectral exactness and spectral stability. In Section \ref{sec: degeneration tri} we prove statements $(iii)$ and $(iv)$ of Theorem \ref{thm: spectral conv tri weak}. Section \ref{sec: homog} is devoted to the proof of Theorem \ref{thm: spectral conv tri weak}$(ii)$, which requires several results from homogenisation theory. In the Appendices we collect some auxiliary results among which the proof of the Triharmonic Green Formula, which is of general interest.

\section{Main result} \label{sec: main}

\subsection{Boundary conditions.}
Given a bounded domain $\Omega \subset \R^N$, we consider the quadratic form defined by
\begin{equation}
\label{def: Q tri}
Q_{\Omega}(u,v) = \int_{\Omega} D^3 u : D^3 v \, dx + \int_{\Omega} u v \, dx,
\end{equation}
for all $u,v \in V(\Omega)$, where $V(\Omega)$ is a linear subspace of $H^3(\Omega)$, $H^3_0(\Omega) \subset V(\Omega)$ and $V$ is complete with respect to the $H^3$-norm. By the second representation theorem
\cite[Theorem 2.23, VI.2]{Kato}, there exists a densely defined, non-negative and self-adjoint operator $A_{V(\Omega)}$ with domain
$\dom(A_{V(\Omega)}) \subset H^3(\Omega)$ such that
\[
Q_{\Omega}(u,v) = (A^{1/2}_{V(\Omega)} u, \, A^{1/2}_{V(\Omega)} v),
\]
for all $u,v \in V(\Omega)$. Assume that the embedding of $V(\Omega)$ in $L^2(\Omega)$ is compact.   Then, $A_{V(\Omega)}$ has compact
resolvent, hence it has purely discrete spectrum, made of an increasing sequence of eigenvalues diverging to $+ \infty$.  Let us consider the
eigenvalue problem
\begin{equation}
\label{def: tri main eigen problem}
\int_{\Omega} D^3 u : D^3 v \, dx + \int_{\Omega} u v \, dx  = \lambda \int_{\Omega} uv \, dx,
\end{equation}
in the unknowns $\lambda$, $u \in V(\Omega)$ for all $v \in V(\Omega)$. We briefly recall the boundary conditions we are interested in. Their
identification is achieved via the Triharmonic Green Formula, stated and proved in Theorem \ref{thm: trih green}. Let $k \in \N$, $0 \leq k
\leq 3$ and let us set $V(\Omega) = H^3(\Omega) \cap H^k_0(\Omega)$. If $k = 3$ then $V(\Omega) = H^3_0(\Omega)$ in \eqref{def: Q tri}.
Formula \eqref{eq: triharmonic green 2} implies that $A_{V(\Omega)}$ is the Dirichlet triharmonic operator associated with
\begin{equation}\label{Dirichlet bc}
\begin{cases}
- \Delta^3 u + u = \lambda u, &\textup{in $\Omega$,}\\
u = \frac{\partial u}{\partial n} = \frac{\partial^2 u}{\partial n^2} = 0, &\textup{on $\partial \Omega$.}
\end{cases}
\end{equation}

When $k=2$, $V(\Omega) = H^3(\Omega) \cap H^2_0(\Omega)$. By \eqref{eq: triharmonic green 2} we deduce that the classical eigenvalue problem
associated with \eqref{def: tri main eigen problem} on $V(\Omega)$ is defined by
\begin{equation}
\label{triharmonic strong}
\begin{cases}
- \Delta^3 u + u = \lambda u, &\textup{in $\Omega$,}\\
u = \frac{\partial u}{\partial n} = 0, &\textup{on $\partial \Omega$,}\\
\frac{\partial^3 u}{\partial n^3} = 0, &\textup{on $\partial \Omega$.}
\end{cases}
\end{equation}
In this case we say that the classical operator $- \Delta^3 u + u$ associated with problem \eqref{triharmonic strong} satisfies
\textit{strong intermediate boundary conditions} on $\partial \Omega$.

Finally, when $k=1$, $V(\Omega) = H^3(\Omega) \cap H^1_0(\Omega)$. By \eqref{eq: triharmonic green 2} we deduce that
\begin{equation}
\label{triharmonic weak}
\begin{cases}
- \Delta^3 u + u = \lambda u, &\textup{in $\Omega$,}\\
u = 0, &\textup{on $\partial \Omega$,}\\
 \big( (n^T D^3u)_{\partial \Omega} : D_{\partial\Omega}n \big) - \frac{\partial^2 (\Delta u)}{\partial n^2} - 2 \Div_{\partial
 \Omega}(D^3u[n\otimes n])_{\partial \Omega} = 0,  &\textup{on $\partial \Omega$,}\\
\frac{\partial^3 u}{\partial n^3} = 0, &\textup{on $\partial \Omega$,}
\end{cases}
\end{equation}
where we have denoted by $( \cdot )_{\partial \Omega}$ the tangential part of a tensor (which can be defined formally exactly as the tangential Hessian, see Def. \ref{tang operators} below) , $D_{\partial\Omega}$ is the tangential Jacobian, $n$ is the outer unit normal to $\partial \Omega$, $\Div_{\partial \Omega}$ is the tangential divergence, $[n\otimes n] = (n_i n_j)_{i,j= 1, \dots, N}$.
In this case, we say that the classical operator $- \Delta^3 u + u$ associated with problem \eqref{triharmonic strong} satisfies
\textit{weak intermediate boundary conditions} on $\partial \Omega$. Note that the curvature tensor $D_{\partial \Omega} n$ appears non-trivially in the second boundary condition. To the best of our knowledge these boundary conditions were never defined before in this form.
%

\subsection{Main theorem}
%

Let $\Omega_\eps$, $\eps > 0$ and $\Omega$ be as in  \eqref{eq: geometric setting}. Set $\Gamma := \overline{W} \times \{0\}$. Let $A_{\Omega_\epsilon}$ be the operator associated to \eqref{triharmonic weak first}, $\eps > 0$, and define $A_\Omega$ in a analogous way by replacing $\Omega_\eps$ with $\Omega$. Let $A_{\Omega, S}$ be the operator associated to
\begin{equation}
\label{lim prob 1}
\begin{cases}
A_{\Omega, S} u := - \Delta^3 u + u = \lambda u, &\textup{in $\Omega$,}\\
(WBC), &\textup{on $\partial \Omega \setminus \Gamma$ }\\
(SBC), &\textup{on $\Gamma$,}\\
\end{cases}
\end{equation}
where $(WBC)$ are the boundary conditions in \eqref{triharmonic weak}, $(SBC)$ those in \eqref{triharmonic strong}. Let $A_{\Omega,D}$ be the operator associated to
\begin{equation}
\label{lim prob 2}
\begin{cases}
A_{\Omega, D}u := - \Delta^3 u + u = \lambda u, &\textup{in $\Omega$,}\\
(WBC), &\textup{on $\partial \Omega \setminus \Gamma$ }\\
(DBC), &\textup{on $\Gamma$,}\\
\end{cases}
\end{equation}
where $(DBC)$ are the Dirichlet boundary conditions defined in \eqref{Dirichlet bc}. Finally, let $\hat{A}_{\Omega}$ be the operator associated to
\begin{equation}
\label{lim prob 3}
\begin{cases}
\hat{A}_{\Omega} u := - \Delta^3 u + u = \lambda u, &\textup{in $\Omega$,}\\
(WBC), &\textup{on $\partial \Omega \setminus \Gamma$ }\\
u=\partial_{x^3_N} u = 0, &\textup{on $\Gamma$}, \\
\Delta (\partial_{x^2_N} u) + 2 \Delta_{N-1} (\partial_{x^2_N} u) + K_1 \partial_{x_N} u = 0, &\textup{on $\Gamma$,}
\end{cases}
\end{equation}
where $K_1 > 0$ is given by
{\small\begin{equation}\label{def: K1}
    \begin{split}
    K_1 &=  \int_Y \Bigg( \Delta^2\Bigg(\frac{\partial V}{\partial y_N}\Bigg) + \Delta_{N-1}\Bigg(\frac{\partial (\Delta V)}{\partial
    y_N}\Bigg) + \Delta^2_{N-1}\Bigg(\frac{\partial V}{\partial y_N} \Bigg) \Bigg) b(\bar{y}) \diff{\bar{y}} \\
    &= \int_{Y\times (-\infty,0)} |D^3 V|^2\,\diff{y},
    \end{split}
\end{equation}
}where the function $V$ is $Y$-periodic in the variables $\bar{y}$ and satisfies the following microscopic problem
\begin{equation}\label{def: V}
\begin{cases}
\Delta^3 V = 0, &\textup{in $Y \times (-\infty, 0)$},\\
V(\bar{y}, 0) = b(\bar{y}), &\textup{on $Y$},\\
- \partial_{y_N^2} (\Delta V) + 2 \partial_{y_N^2}(\Delta_{N-1}V) = 0, &\textup{on $Y$},\\
\partial_{y_N^3} V= 0, &\textup{on $Y$}.
\end{cases}
\end{equation}
Then we have the following

\begin{theorem}
\label{thm: spectral conv tri weak}
For $\eps \geq 0$ let $\Omega_\eps \subset \R^N$ be defined by \eqref{eq: geometric setting}. Let $A_{\Omega_\epsilon}$, $\eps > 0$, $A_\Omega$, $A_{\Omega, S}$, $A_{\Omega, D}$, $\hat{A}_{\Omega}$ be the operators defined above in \eqref{lim prob 1}, \eqref{lim prob 2}, \eqref{lim prob 3}. Then:
\begin{enumerate}[label=(\roman*)]
\item \textup{[Spectral stability]} If $\alpha > 5/2$, then $A^{-1}_{\Omega_\epsilon} \overset{\mathcal{C}}{\rightarrow} A^{-1}_{\Omega}$.
\item \textup{[Strange term]} If $\alpha = 5/2$, then $A^{-1}_{\Omega_\epsilon} \overset{\mathcal{C}}{\rightarrow} \hat{A}^{-1}_{\Omega}$.
\item \textup{[Mild instability]} If $3/2 <\alpha < 5/2$, then $A^{-1}_{\Omega_\epsilon} \overset{\mathcal{C}}{\rightarrow}
    A^{-1}_{\Omega, S}$.
\item \textup{[Strong instability]} If $\alpha \leq 1$, then $A^{-1}_{\Omega_\epsilon} \overset{\mathcal{C}}{\rightarrow}
    A^{-1}_{\Omega,D}$.
\end{enumerate}
In particular, the eigenvalues $\la_j(\Omega_\eps)$, $j \geq 1$, of \eqref{triharmonic weak first} converge as $\eps \to 0$ to the eigenvalues of $A_{\Omega}$ in case $(i)$, $\hat{A}_{\Omega}$ in case $(ii)$, $A_{\Omega, S}$ in case $(iii)$ and $A_{\Omega,D}$ in case $(iv)$.
\end{theorem}

The compact convergence $\Ccon$ in the previous theorem is defined in Definition \ref{def: E C conv}. \\
The novelty of Theorem \ref{thm: spectral conv tri weak} lies in the identification of the double instability effect, namely a first degeneration to SBC when $\alpha \in (\frac{3}{2}, \frac{5}{2})$ and a further degeneration to Dirichlet when $\alpha \leq 1$. We immediately give a proof of item $(i)$:

\begin{proof}[Proof of Thm \ref{thm: spectral conv tri weak}(i)]
The follows from \cite[Theorem 4]{FerLamb}, with $m = 3$, $k = 1$.
\end{proof}

\begin{remark}
The results in Theorem \ref{thm: spectral conv tri weak} can be easily generalised to the case where $\Omega$ has a piecewise flat boundary,
$\Omega_\eps$, $\Omega$ belong to the same atlas class in the sense of \cite[Definition 2.4]{Burla} for all $\eps >0$, and have an oscillating boundary locally described by \eqref{eq: geometric setting}. Indeed, if $V$ is one of the chart in the common atlas class, using a partition of unity we can directly assume that
\[
\Omega_\eps \cap V = \{ (\bar{x}, x_N) \in \R^N : \bar{x} \in W, \, -1 < x_N < g_{V,\epsilon}(\bar{x}) \}.
\]
It is clear that if we allow $\alpha > 0$ to be chart dependent, we may find a limiting boundary value problem with mixed boundary conditions.
Nevertheless, the passage to the limit can be treated locally exactly as in Theorem \ref{thm: spectral conv tri weak}. As an
example, assume that the sequence of open sets $(\Omega_\eps)_{\eps>0}$ has a common atlas given by three charts $V_1$, $V_2$ and $V_3$. Then
up to a possible rotation and translation
\[
\Omega_\eps \cap V_1 = \{(\bar{x}, x_N) \in \R^N : \bar{x} \in W, \, -1 < x_N < \eps^{\alpha_1} b_1(\bar{x}/\eps)\},
\]
\[
\Omega_\eps \cap V_2 = \{(\bar{x}, x_N) \in \R^N : \bar{x} \in W, \, -1 < x_N < \eps^{\alpha_2} b_2(\bar{x}/\eps)\},
\]
\[
\Omega_\eps \cap V_3 = \Omega \cap V_3, \quad \eps > 0,
\]
with $\alpha_1 > 5/2$, $\alpha_2 \leq 1$. Then the limiting boundary value problem in $\Omega$ will be in the form
\[
\begin{cases}
-\Delta^3 u + u = \la u, \quad &\textup{in $\Omega$,} \\
(WBC), \quad &\textup{in $\Gamma_1 \cup \Gamma_3$,} \\
(DBC), \quad &\textup{in $\Gamma_2$,}
\end{cases}
\]
where $\partial \Omega = \Gamma_1 \cup \Gamma_2 \cup \Gamma_3$, $\Gamma_j$ being the boundary of $\Omega$ inside $V_j$, $j = 1,2,3$.
\end{remark}

\begin{remark} \label{rem: open}
The case $\alpha \in (1,3/2)$ in Theorem \ref{thm: spectral conv tri weak} remains at the moment open. The proof of Theorem \ref{thm: spectral conv tri weak} seems to suggest that $\alpha = 3/2$ is not a critical threshold; in other words, we do not expect degeneration to the Dirichlet problem at $\alpha = 3/2$. The main difficulty is that the derivatives of $T_\eps \varphi$, $\varphi \in L^2(\Omega)$, where $T_\eps$ is the pullback operator defined in \eqref{def: T_eps} have singularities that are balanced by neither the shrinking of the set $\Omega_\eps$ when $\alpha \in (1, 3/2)$, nor by the vanishing of the traces of the eigenfunctions at the boundary. The construction of a more efficient extension operator $T_\eps : H^3(\Omega) \cap H^1_0(\Omega) \to H^3(\Omega_\eps) \cap H^1_0(\Omega_\eps)$ is however even more challenging: note that the classically used Sobolev extension operators do not work here since they do not preserve the boundary conditions.
\end{remark}

\section{Auxiliary results} \label{sec: auxil}

\noindent \textbf{ $\bullet$ A diffeomorphism between $\Omega$ and $\Omega_\eps$.}\\
Let us define a diffeomorphism $\Phi_\epsilon$ from $\Omega_\eps$ to $\Omega$ by
\begin{equation}
\label{def: Phi_eps}
\Phi_\eps( \bar{x}, x_N) = (\bar{x}, x_N - h_\eps(\bar{x}, x_N)), \quad \textup{for all $x=(\bar{x}, x_N) \in \Omega_\eps$,}
\end{equation}
where $h_\epsilon$ is defined by
\[
h_\epsilon (\bar{x}, x_N) =
\begin{cases}
0, &\textup{if $-1 \leq x_N \leq -\epsilon$},\\
g_\epsilon(\bar{x})\Big(\frac{x_N+\epsilon}{g_\epsilon(\bar{x})+\epsilon}\Big)^{4}, &\textup{if $-\epsilon \leq x_N \leq
g_\epsilon(\bar{x})$}.
\end{cases}
\]
By standard calculus one can prove the following

\begin{lemma}
\label{lemma: h_eps}
The map $\Phi_\epsilon$ is a diffeomorphism of class $C^3$ and there exists a constant
$c > 0$ independent of $\epsilon$ such that
$
\abs{h_\epsilon} \leq c \epsilon^\alpha$ and $\left \lvert D^l h_\epsilon \right \rvert \leq c \epsilon^{\alpha - l}$,
for all $l= 1, \dots, 3$, $\epsilon > 0$ sufficiently small.
\end{lemma}

We then introduce the pullback operator
\begin{equation}
\label{def: T_eps}
T_\epsilon : L^2(\Omega) \to L^2(\Omega_\epsilon), \quad T_\epsilon u = u \circ \Phi_\epsilon
\end{equation}
for all $u\in L^2(\Omega)$.\\
\ \, \\

\noindent \textbf{$\bullet$ Unfolding method.}\\
We recall the following notation and results from \cite{FerLamb} regarding the unfolding method. For any $k \in \Z^{N-1}$ and $\eps > 0$ we define
\begin{equation}
\label{def: anistropic unfolding cell}
\left\{
\begin{aligned}
&C^k_\eps = \eps k + \eps Y,\\
&I_{W,\eps} = \{k \in \Z^{N-1} : C^k_\eps \subset W\},\\
&\widehat{W}_\eps = \bigcup_{k \in I_{W, \eps}} C^k_\eps.
\end{aligned}
\right.
\end{equation}

\begin{definition}
\label{def: anisotropic unfolding}
Let $u$ be a real-valued function defined in $\Omega$. For any $\epsilon > 0$ sufficiently small the unfolding $\hat{u}$ of $u$ is the
real-valued function defined on $\widehat{W}_\epsilon \times Y \times (-1/\epsilon, 0)$ by
\[
\hat{u}(\bar{x}, \bar{y}, y_N) = u\Big( \epsilon\Big[\frac{\bar{x}}{\epsilon}\Big] + \epsilon \bar{y}, \epsilon y_N \Big),
\]
for almost all $(\bar{x}, \bar{y}, y_N)) \in \widehat{W}_\epsilon \times Y \times (-1/\epsilon, 0)$, where
$[\bar{x}\epsilon^{-1}]$ denotes the integer part of the vector $\bar{x} \epsilon^{-1}$ with respect to $Y$, i.e., $[\bar{x}
\epsilon^{-1}] = k$ if and only if $\bar{x} \in C^k_\epsilon$.
\end{definition}

The following lemma will be often used in the sequel. For a proof we refer to  \cite[Proposition 2.5(i)]{CioDamGri2}.

\begin{lemma}
\label{lemma: exact int formula}
Let $a \in [-1,0[$ be fixed. Then
\begin{equation}
\label{eq: exact int formula}
\int_{\widehat{W}_\eps \times (a,0)} u(x) dx = \eps \int_{\widehat{W}_\eps \times Y \times (a/\eps,0)} \hat{u}(\bar{x},y) d\bar{x} dy
\end{equation}
for all $u \in L^1(\Omega)$ and $\eps>0$ sufficiently small. Moreover
\begin{equation*}
\int_{\widehat{W}_\eps \times (a,0)} \left|\frac{\partial^l u(x)}{\partial x_{i_1} \cdots \partial x_{i_l}} \right|^2 dx = \eps^{1-2l}
\int_{\widehat{W}_\eps \times Y \times (a/\eps,0)} \left| \frac{\partial^l \hat{u}}{\partial y_{i_1} \cdots \partial y_{i_l}}(\bar{x},y)
d\bar{x} \right|^2dy,
\end{equation*}
for all $l \leq 3$,  $u \in H^{3}(\Omega)$ and $\eps>0$ sufficiently small.
\end{lemma}

Let $H^3_{{\rm Per}_Y, {\rm loc}}(Y \times (-\infty,0))$ be the subspace of $H^3_{\rm loc}(\R^{N-1} \times (-\infty,0))$ containing
$Y$-periodic functions in the first $(N-1)$ variables $\bar{y}$. We then define $H^3_{loc}(Y \times (-\infty,0))$ to be the space of
functions in
$H^3_{{\rm Per}_Y, {\rm loc}}(Y \times (-\infty,0))$ restricted to $Y \times (-\infty,0)$. Finally we set
\begin{multline}
w^{3,2}_{{\rm Per}_Y}(Y \times (-\infty, 0)) := \big\{ u \in H^3_{{\rm Per}_Y, {\rm loc}}(Y \times (-\infty,0))\\
: \norma{D^\gamma u}_{L^2(Y \times (-\infty,0))} < \infty, \forall |\gamma| = 3 \big\}.
\end{multline}
For any $d<0$, let $\mathcal{P}_{hom,y}^{l}(Y \times (d,0))$ be the space of homogeneous polynomials of degree at most $l$ restricted to the
domain $(Y \times (d,0))$. Let $\eps > 0$ be fixed. We define the projectors $P_i$ from $L^2(\widehat{W}_\eps, H^3(Y \times (-1/\eps ,
0)))$ to $L^2(\widehat{W}_\eps, \mathcal{P}_{hom,y}^{i}(-1/\eps, 0))$ by setting
\begin{equation*}
P_i (\psi )= \sum_{|\eta|=i} \int_{Y} D^\eta \psi(\bar{x},\bar{\zeta},0) d\bar{\zeta} \frac{y^\eta}{\eta!}
\end{equation*}
for all $i = 0, 1,2$. We now set $Q_{2} = P_{2}$, $Q_{1} = P_{1} (\mathbb{I} - Q_{1})$,$Q_{0} =
P_0\big(\mathbb{I} - \sum_{j=1}^{2} Q_j\big)$. Note that $Q_{3-j}$, $j=1, \dots, 3$ is a projection on the space of homogeneous  polynomials
of degree $3-j$, with the property that $Q_{3-k} (p) = 0$ for all polynomials $p$ of degree  $3-k$ with $k\ne j$. We finally set
\begin{equation}
\label{def: p}
\p= Q_0 + Q_1 + Q_2,
\end{equation}
which is a projector on the space of polynomials in $y$ of degree at most $2$. Note that $D_y^\beta \p(\psi )(\bar{x}, \bar{y}, 0) =
\int_{Y}D_y^\beta\psi (\bar{x}, \bar{y}, 0) d\bar y$ for all  $|\beta| = 0,\dots, 2$.

\begin{lemma}
\label{lemma: unfolding convergence poly}
The following statements hold:
\begin{enumerate}[label=(\roman*)]
\item Let $v_\epsilon \in H^3(\Omega)$ with $\norma{\hat {v_\epsilon}}_{H^3(\Omega)} \leq M$, for all $\epsilon > 0$. Let
    $V_\epsilon$ be defined by
\[
\begin{split}
V_\epsilon(\bar{x}, y) =  & \hat{v_\epsilon}(\bar{x}, y) - \p(v_\epsilon)(\bar{x}, y),
\end{split}
\]
for $(\bar{x}, y) \in \widehat{W_\epsilon}\times Y\times (-1/\epsilon, 0)$, where $\p$ is defined by \eqref{def: p} . Then there exists a
function $\hat{v}\in L^2(W, w^{3,2}_{\textup{Per}_Y}(Y\times (-\infty,0)))$ such that for every $d<0$
\begin{enumerate}[label=(\alph*)]
\item $\frac{D_y^{\gamma}V_\epsilon}{\epsilon^{5/2}} \rightharpoonup D_y^{\gamma}\hat{v}$ in $L^2(W\times Y \times (d,0))$ as $\epsilon
    \to 0$, for any $\gamma \in \N_0^N$, $\abs{\gamma} \leq 2$.
\item $\frac{D_y^{\gamma}V_\epsilon}{\epsilon^{5/2}} \rightharpoonup D_y^{\gamma}\hat{v}$ in $L^2(W\times Y \times (-\infty,0))$ as
    $\epsilon \to 0$, for any $\gamma \in \N_0^N$, $\abs{\gamma} = 3$,
\end{enumerate}
where it is understood that the functions $V_\epsilon, D_y^{\gamma}V_\epsilon$ are extended by zero to the whole of $W\times Y \times
(-\infty,0)$ outside their natural domain of definition $\widehat{W_\epsilon}\times Y\times (-1/\epsilon, 0) $.
\item If $\psi\in W^{1,2}(\Omega)$, then $\lim_{\eps \to 0} \widehat{(T_\epsilon \psi)_{|\Omega}} =  \psi(\bar{x},0)$ in $L^2(W \times Y
    \times (-1,0))$.
\end{enumerate}
\end{lemma}
\ \, \\

\noindent \textbf{$\bullet$ Tangential Calculus.}\\
Recall now the following standard definitions of the tangential differential operators.  We refer to \cite[Chapter 9]{DelZol} for details and
further information. Given $A \subset \R^N$ let $d_A$ be the Euclidean distance function from $A$, defined by $d_A(x) = \inf_{y \in A} |x-y|$.
We define the oriented distance function $b_A$ from $A$ by
\[
b_A(x) = d_A(x) - d_{A^C}(x),
\]
for all $x \in \R^N$. Let now $\Omega$ be an bounded open set of class $C^2$. In this case $b_\Omega$ coincides with the signed distance from
$\partial \Omega$. It is well-known that there exists $h > 0$ and a tubular neighbourhood $S_{2h}(\partial \Omega)$ of radius $h$ such that
$b_\Omega \in C^2(S_{2h}(\partial \Omega))$, see \cite{MR110078}. We define the projection of a point $x$ to $\partial \Omega$ by
\begin{equation}
\label{def: p orth}
p(x) = x - b_{\Omega}(x) \nabla b_{\Omega}(x),
\end{equation}
for all $x \in S_{2h}(\partial \Omega)$. If $f \in C^0(\partial \Omega)$ we write $(f)_{\partial \Omega} = (f \circ p)|_{\partial \Omega}$.

\begin{definition}
\label{tangential gradient}
Let $\Omega$ be an bounded open set of class $C^2$ and let $h>0$ be such that $b_\Omega \in C^2(S_{2h}(\partial \Omega))$. Let $f \in
C^1(\partial \Omega)$ and let $F \in C^1(S_{2h}(\partial \Omega))$ be a $C^1$ extension of $f$ to $S_{2h}(\partial \Omega)$ (that is,
$F|_{\partial \Omega} = f$). We define the tangential gradient of $f$ on $\partial \Omega$ by
\[
\nabla_{\partial \Omega} f = \nabla F|_{\partial \Omega} - \frac{\partial F}{\partial n} n.
\]
\end{definition}
%

\begin{definition} \label{tang operators}
Let $N \geq 1$, $v \in C^1(\partial \Omega)^N$. We define the tangential Jacobian matrix of $v$ by $D_{\partial \Omega} v = D(v \circ
p)|_{\partial \Omega}$ and the tangential divergence of $v$ by $\Div_{\partial \Omega} (v \circ p)|_{\partial \Omega} = {\rm tr} ( D_{\partial
\Omega} v )$. Assume now $\Omega$ is of class $C^3$ and $f \in C^2(\partial \Omega)$. We define the Laplace-Beltrami operator of $f$ by
\[
\Delta_{\partial \Omega} f = \Delta(f \circ p)|_{\partial \Omega} = \Div_{\partial \Omega}(\nabla_{\partial \Omega} f),
\]
and similarly we define the tangential Hessian matrix by $D^2_{\partial \Omega} f = D_{\partial \Omega}(\nabla_{\partial \Omega} f)$.
\end{definition}
%
%
%
%

We conclude this section recalling the following important
\begin{theorem}[Tangential Divergence Theorem]
\label{thm: tangential div thm}
Let $\Omega$ be a bounded open set of class $C^2$ and let $v \in C^1(\partial \Omega)^N$. Let $\mathcal{H}$ be the trace of the second
fundamental form of $\partial \Omega$.  Then
\begin{equation}
\label{Tangential Divergence Thm}
\int_{\partial \Omega} \Div_{\partial \Omega}\! v \,\,dS = \int_{\partial\Omega} \mathcal{H}\, (v \cdot n) \, dS.
\end{equation}
Let $f \in C^1(\partial \Omega)$. Then
\begin{equation}
\label{Tangential Green's formula}
\int_{\partial \Omega} (f \Div_{\partial \Omega}\! v + \nabla_{\partial \Omega}f \cdot v) \, dS = \int_{\partial \Omega} \mathcal{H}\, f \, (v
\cdot n)\, dS.
\end{equation}
\end{theorem}
\begin{proof}
We refer to \cite[\S5.5 Chapter 9]{DelZol}.
\end{proof}
%

\section{Spectral exactness and spectral stability} \label{sec: spect stab}
Let $(\cH_\eps)_{\eps \in [0,1]}$ be a family of Hilbert spaces. Let $(\cE_{\!\!\eps})_{\eps \in (0,1]}$ be a connecting system for $(\cH_\eps)_{\eps \in [0,1]}$, that is, $\cE_{\!\!\eps} \in L(\cH_0, \cH_\eps)$, $\eps \in (0,1]$, and
\[
\lim_{\eps \to 0} \norma{\cE_{\!\!\eps} u}_{\cH_\eps} = \norma{u}_{\cH_0}
\]
for every $u \in \cH_0$.

We recall the following definitions.

\begin{definition} \label{def: E C conv} Let $(\cH_\eps)_{\eps\in[0,1]}$ and $\cE_{\!\!\eps}$ be as above.
\begin{enumerate}[label =(\roman*)]
\item Let $u_\eps\in   \cH_\eps$, $\eps >0$.  We say that $u_\eps$ $\cE$-converges to $u$ as $\eps \to 0$ if $\norma{u_\eps - \cE_{\!\!\eps} u}_{\cH_\eps} \to 0$ as $\eps \to 0$. We write $u_\eps \overset{\cE}{\longrightarrow} u$.
\item Let $ B_\eps \in L(\cH_\eps)$,  $\eps >0$. We say that $B_\eps$ $\cE\cE$-converges to a linear operator $B_0 \in L(\cH_0)$ if $B_\eps u_\eps \overset{\cE}{\longrightarrow} B_0 u$ whenever $u_\eps \overset{\cE}{\longrightarrow} u\in \cH_0$. We write $B_\eps \overset{\cE\cE}{\longrightarrow} B_0$.
\item Let $ B_\eps \in L(\cH_\eps)$, $  \eps >0$. We say that $B_\eps$ compactly converges to $B_0 \in L(\cH_0)$ (and we write $B_\eps \Ccon B_0$) if the following two conditions are satisfied:
    \begin{enumerate}[label=(\alph*)]
    \item $B_\eps \overset{\cE\cE}{\longrightarrow} B_0$ as $\eps \to 0$;
    \item for any family $u_\eps \in \cH_{\eps}$, $\eps>0$, such that $\norma{u_\eps}_{\cH_\eps}=1$ for all $\eps \in (0,1)$, there exists a subsequence $B_{\eps_k}u_{\eps_k}$ of $B_\eps u_\eps$ and $\bar{u} \in \cH_0$ such that $B_{\eps_k}u_{\eps_k} \overset{\cE}{\longrightarrow} \bar{u}$ as $k \to \infty$.
    \end{enumerate}
\end{enumerate}
\end{definition}

\begin{definition}
Let $T$, $T_n$ be closed operators in $\cH$, $\cH_n$ respectively, $n \in \N$.
\begin{enumerate}
\item The sequence $(T_n)_{n \in \N}$ is called spectrally inclusive if for every $\la \in \sigma(T)$, there exists a sequence $(\la_n)_{n \in \N}$, $\la_n \in \sigma(T_n)$, $n \in \N$ such that $\la_n \to \la$.
\item We say that spectral pollution occurs for $(T_n)_{n \in \N}$ if there exists $\la \in \rho(T)$ and $\la_n \in \sigma(T_n)$, $n \in \N$ such that $\la_n \to \la$.
\item The sequence $(T_n)_{n \in \N}$ is called spectrally exact if it is spectrally inclusive and no spectral pollution occurs.
\end{enumerate}
\end{definition}

Let $(A_\eps)_{\eps \in [0,1]}$ be a family of closed densely defined linear operators, $A_\eps \in C(\cH_\eps)$, $\eps \in [0,1]$. Assume that:
\begin{align*} &\textup{(A1)(a): $\exists \la_0 \in \bigcap_{\eps \in [0,1]} \rho(A_\eps)$, $(A_\eps - \la_0)^{-1}$  compact $\eps \in [0,1]$,} \\
&\textup{(A1)(b) $(A_\eps - \la_0)^{-1} \Ccon (A_0 - \la_0)^{-1}$ as $\eps \to 0$.}
\end{align*}
 Then \cite[Theorem 2.6]{Boegli2} implies that $(A_\eps)_{\eps \in (0,1]}$ is a spectrally exact approximation of $A_0$.
Consider now the following setting. Let $m \in \N$, and let $\Omega$ be an open set of $\R^N$. Let $M$ be the number of multi-indices $\alpha = (\alpha_1, \dots, \alpha_N) \in \N_0^N$ with length $|\alpha| = |\alpha_1| + \cdots + |\alpha_N| = m$. For all $\alpha, \beta \in \N_0^N$  with $|\alpha| = |\beta| = m$ let $c_{\alpha\beta}$ be bounded measureable real-valued functions defined on $\R^N$, $c_{\alpha \beta} = c_{\beta \alpha}$ such that
\[
\sum_{|\alpha| = |\beta| = m} c_{\alpha\beta}(x) \xi_\alpha \xi_\beta \geq 0,
\]
for all $x \in \R^N$, $(\xi_\alpha)_{|\alpha| = m} \in \R^m$. For all measurable open sets $\Omega \in \R^N$ we define
\begin{equation} \label{Q form}
Q_{\Omega}(u,v) = \int_{\Omega} \bigg( c_{\alpha\beta} D^\alpha u D^\beta v + uv \bigg) dx
\end{equation}
Let $V(\Omega)$ be a linear subspace of $H^m(\Omega)$ containing $H^m_0(\Omega)$. Assume that $V(\Omega)$ endowed with the norm $Q_{\Omega}(\cdot)^{1/2}$ is complete. Then there exists a unique self-adjoint operator $A_{V(\Omega)}$ such that
\begin{equation} \label{eq: def T}
Q_{\Omega}(u,v) = (A^{1/2}_{V(\Omega)} u, A^{1/2}_{V(\Omega)} v)
\end{equation}
for all $u,v \in V(\Omega)$.

For $\eps \geq 0$, let $\Omega_\eps$ be a bounded domain of $\R^N$. In this setting we can give the following
\begin{definition}\label{spectral stability}
Let $W(\Omega_0)$ be a linear subspace of $H^m(\Omega_0)$ containing $H^m_0(\Omega_0)$. Assume that $W(\Omega_0)$ endowed with the norm $Q_{\Omega_0}^{1/2}$ is complete. The sequence of operators $(A_{V(\Omega_\eps)})_{\eps > 0} \cup \{ A_{W(\Omega_0)}\}$, defined as in \eqref{eq: def T} with $\Omega$ replaced by $\Omega_\eps$, and $Q_{\Omega_\eps}$ as in \eqref{Q form} for all $\eps > 0$, is said to be spectrally stable if $(A_{V(\Omega_\eps)})_{\eps > 0}$ is a spectrally exact approximation of $A_{W(\Omega_0)}$ and $W(\Omega_0) = V(\Omega_0)$.
\end{definition}

With Definition \ref{spectral stability}, \cite[Theorem 3.5]{ArrLamb} can be rephrased as:

\begin{theorem} \label{thm: (C) implies Ccon}
Assume that Condition (C), see \cite[Definition 3.1]{ArrLamb}, is satisfied by the sequence of operators $A_{V(\Omega_\eps)}$, $A_{V(\Omega)}$ associated with the quadratic forms $Q_{\Omega_\eps}$, $Q_\Omega$. Then the sequence of operators $(A_{V(\Omega_\eps)})_{\eps > 0}$ is spectrally stable.
\end{theorem}


\section{Proof of Theorem \ref{thm: spectral conv tri weak}(iii), (iv)}
\label{sec: degeneration tri}

To prove $(iii)$ and $(iv)$ in Theorem \ref{thm: spectral conv tri weak}, we will show that Condition $(C)$, see \cite[Definition 3.1]{ArrLamb}, holds for the operators $A_{\Omega_\eps}$ associated to \eqref{triharmonic weak first}. An application of Theorem \ref{thm: (C) implies Ccon} will then prove the claims.

Establishing Condition $(C)$ will require several lemmata. We first establish a general lemma concerning the limiting boundary behaviour of sequences $(u_\eps)_\eps$ such that $u_\eps \in H^3(\Omega_\eps) \cap H^1_0(\Omega_\eps)$ and $\norma{u_\eps}_{H^3(\Omega_\eps)} < \infty$, for all $\eps > 0$.

For $\eps > 0$, we define
\[
\Omega^\sharp_\eps = \{ (\bar{x}, x_N) \in \R^N : \bar{x} \in \overline{W}, -1 \leq x_N < g_\eps(\bar{x}) \}, \quad \Omega^\sharp = \ov{W} \times
[-1,0),
\]
and for any $l \in \N$, $\eps > 0$ we set
\[H^{l}_{0,*}(\Omega_\eps) = \ov{C^{\infty}_c(\Omega_\eps^\sharp)}^{H^l(\Omega_\eps)}, \quad H^{l}_{0,*}(\Omega) =
\ov{C^{\infty}_c(\Omega^\sharp)}^{H^l(\Omega)}.\]

In the case of sequence of functions in $(u_\eps)_{\eps > 0}$, $u_\eps \in H^3(\Omega_\eps) \cap H^1_0(\Omega_\eps)$, we have the following
result

\begin{lemma}
\label{lemma: degeneration H^3 H^1_0}
Let $Y =[-1/2, 1/2]^{N-1}$, $\alpha \in \R$, $\alpha > 0$. Let $\Omega = W \times (-1,0)$, where $W \subset \R^{N-1}$ is bounded domain of
class $C^3$. Let $\Omega_\eps$ be as in \eqref{eq: geometric setting}. Let $(u_\eps)_{\eps > 0}$ be such that $H^3(\Omega_\eps) \cap
H^1_{0,*}(\Omega_\eps)$ for all $\eps > 0$ and $u_\eps|_{\Omega} \to u$ weakly in $H^3(\Omega)$. Let also $\hat{u}\in L^2(W, H^3(Y \times
(-1,0)))$ be defined by \eqref{limit: U_eps convergence}. Then:
\begin{enumerate}[label=(\roman*)]
\item If $\alpha > 5/2$ then $u \in H^3(\Omega) \cap H^1_{0,*}(\Omega)$;
\item If $\alpha = 5/2$ then $u \in H^3(\Omega) \cap H^1_{0,*}(\Omega)$ and for $i,j \in \{1, \dotsc, N-1\}$,
\begin{equation}\label{eq: limit case hatu eq} \frac{\partial^2 \hat{u}}{\partial y_i \partial y_j}(\bar{x}, \bar{y},0) = -\frac{\partial u}{\partial x_N}(\bar{x}, \bar{y},0)
\frac{\partial^2 b(\bar{y})}{\partial y_i \partial y_j}.
\end{equation}
\item If $0< \alpha < 5/2$ then $u \in H^3(\Omega)\cap H^2_{0,*}(\Omega)$;
\item If $0 < \alpha \leq 1$ then $u \in H^3_{0,*}(\Omega)$
\end{enumerate}
\end{lemma}
\begin{proof}
\noindent Fix $0< \eps < 1$. We find convenient to treat first the case $\alpha \geq 3/2$. Since $u_\eps \in H_0^1(\Omega_\eps)$
\begin{equation}
\label{eq: boundary equality}
u_\eps(\bar{x}, g_\eps(\bar{x})) = 0, \quad \textup{for a.e. $\bar{x} \in W$. }
\end{equation}
Note that the function $u_\eps(\cdot, g_\eps(\cdot)) \in H^{5/2}(W) \subset H^2(W)$. Differentiation \eqref{eq: boundary
equality} with respect to $x_i$ and then with respect to $x_j$, $ i, j \in \{1, \dots, N-1 \}$ gives
\begin{equation}
\label{eq: boundary equality 2}
\begin{split}
&\frac{\partial^2 u_\eps}{\partial x_i \partial x_j}(\bar{x}, g_\eps(\bar{x})) + \frac{\partial^2 u_\eps}{\partial x_i\partial x_N}(\bar{x},
g_\eps(\bar{x})) \frac{\partial g_\eps(\bar{x})}{\partial x_j} + \frac{\partial^2 u_\eps}{\partial x_j\partial x_N}(\bar{x}, g_\eps(\bar{x}))
\frac{\partial g_\eps(\bar{x})}{\partial x_i}\\
&+ \frac{\partial^2 u_\eps}{\partial x^2_N}(\bar{x}, g_\eps(\bar{x})) \frac{\partial g_\eps(\bar{x})}{\partial x_i}\frac{\partial
g_\eps(\bar{x})}{\partial x_j} + \frac{\partial u_\eps}{\partial x_N}(\bar{x}, g_\eps(\bar{x}))\frac{\partial^2 g_\eps(\bar{x})}{\partial x_i
\partial x_j} = 0,
\end{split}
\end{equation}
for a.e. $\bar{x} \in W$. For $v \in H^1(\Omega_\eps)$, let $\hat{v}(\bar{x},y)$ for all $\bar{x} \in \widehat{W}_\eps$, $\bar{y} \in Y$, $y_N \in (-1/\eps, \eps^{\alpha -1}b(\bar{y}))$ be as in Definition \ref{def: anisotropic unfolding}. It is understood that $\hat{v}$ is set to be zero for all $\bar{x} \in W \setminus \widehat{W}_\eps$.\\
To shorten the notation, define $y_\eps := \eps^{\alpha-1}b(\bar{y})$, $\eps > 0$, $\bar{y} \in W$, and note that by periodicity of $b$, $b(\bar{y}) = b([\bar{x}/\eps] + \bar{y}) = \eps^{-\alpha}\widehat{g_\eps}(\bar{x}, \bar{y})$ for all $(\bar{x}, \bar{y}) \in C^k_\eps \times Y$.

An application of the unfolding operator to equality \eqref{eq: boundary equality 2}, with the help of Lemma \ref{lemma: exact int formula} gives
\[
\begin{split}
&\frac{1}{\eps^2}\frac{\partial^2 \hat{u}_\eps}{\partial y_i \partial y_j}(\bar{x}, \bar{y},  y_\eps) +
\frac{\eps^{\alpha-1}}{\eps^2}\frac{\partial^2 \hat{u}_\eps}{\partial y_i\partial y_N}(\bar{x}, \bar{y}, y_\eps) \frac{\partial
b(\bar{y})}{\partial y_j} + \frac{\eps^{\alpha-1}}{\eps^2}\frac{\partial^2 \hat{u}_\eps}{\partial y_j\partial y_N}(\bar{x}, \bar{y}, y_\eps)
\frac{\partial b(\bar{y})}{\partial y_i}\\
&+ \frac{\eps^{2\alpha-2}}{\eps^2}\frac{\partial^2 \hat{u}_\eps}{\partial y^2_N}(\bar{x}, \bar{y}, y_\eps) \frac{\partial b(\bar{y})}{\partial
y_i}\frac{\partial b(\bar{y})}{\partial y_j} + \frac{\eps^{\alpha-2}}{\eps} \frac{\partial \hat{u}_\eps}{\partial y_N}(\bar{x}, \bar{y},
y_\eps)\frac{\partial^2 b(\bar{y})}{\partial y_i \partial y_j} = 0,
\end{split}
\]
for a.e. $\bar{x} \in W$, for a.e. $\bar{y} \in Y$. Define
\[
\begin{split}
\hat{\Psi}_\eps(\bar{x},y) = &\frac{1}{\eps^2}\frac{\partial^2 \hat{u}_\eps}{\partial y_i \partial y_j}(\bar{x}, y) +
\frac{\eps^{\alpha-1}}{\eps^2}\frac{\partial^2 \hat{u}_\eps}{\partial y_i\partial y_N}(\bar{x}, y) \frac{\partial b(\bar{y})}{\partial y_j} +
\frac{\eps^{\alpha-1}}{\eps^2}\frac{\partial^2 \hat{u}_\eps}{\partial y_j\partial y_N}(\bar{x}, y) \frac{\partial b(\bar{y})}{\partial y_i}\\
&+ \frac{\eps^{2\alpha-2}}{\eps^2}\frac{\partial^2 \hat{u}_\eps}{\partial y^2_N}(\bar{x}, y) \frac{\partial b(\bar{y})}{\partial
y_i}\frac{\partial b(\bar{y})}{\partial y_j} + \frac{\eps^{\alpha-2}}{\eps} \frac{\partial \hat{u}_\eps}{\partial y_N}(\bar{x},
y)\frac{\partial^2 b(\bar{y})}{\partial y_i \partial y_j},
\end{split}
\]
for a.e. $\bar{x} \in W$, for a.e. $\bar{y} \in Y$. Let also $\hat{Y} := \{ y\in \R^N : \bar{y} \in Y, -1 < y_N < \eps^{\alpha - 1} b(\bar{y}) \}$. Then $\hat{\Psi}_\eps \in L^2(W, H^1(\hat{Y}))$.

Since $\hat{\Psi}_\eps(\bar{x},y, y_\eps) = 0$ we have that $ |\hat{\Psi}_\eps(\bar{x},\bar{y}, 0)| \leq \int_0^{y_\eps} |\partial_{y_N}
\hat{\Psi}_\eps (\bar{x}, \bar{y}, t) |\, dt$ for a.e. $\bar{x}\in W$, $\bar{y} \in Y$, from which we deduce
\begin{multline}
\label{ineq: conti psi_eps}
|\hat{\Psi}_\eps(\bar{x},\bar{y}, 0)|\leq \big(\eps^{\alpha-1} \norma{b}_{\infty}\big)^{1/2} \Bigg[ \frac{1}{\eps^2} \bigg\lVert
\frac{\partial^3 \hat{u}_\eps}{\partial y_i \partial y_j \partial y_N}(\bar{x}, \bar{y}, \cdot) \bigg\rVert_{L^2(0, y_\eps)}\\
+ \frac{\eps^{\alpha-1}}{\eps^2} \norma{\nabla b}_{\infty} \bigg\lVert \frac{\partial^3 \hat{u}_\eps}{\partial y_i \partial y^2_N}(\bar{x},
\bar{y}, \cdot) \bigg\rVert_{L^2(0, y_\eps)} + \frac{\eps^{\alpha-1}}{\eps^2} \norma{\nabla b}_{\infty} \bigg\lVert \frac{\partial^3
\hat{u}_\eps}{\partial y_j \partial y^2_N}(\bar{x}, \bar{y}, \cdot) \bigg\rVert_{L^2(0, y_\eps)}\\
+ \frac{\eps^{2\alpha-2}}{\eps^2} \norma{\nabla b}^2_{\infty} \bigg\lVert \frac{\partial^3 \hat{u}_\eps}{\partial y^3_N}(\bar{x}, \bar{y},
\cdot) \bigg\rVert_{L^2(0, y_\eps)}+ \frac{\eps^{\alpha-2}}{\eps} \norma{D^2 b}_{\infty} \bigg\lVert \frac{\partial^2 \hat{u}_\eps}{\partial
y^2_N}(\bar{x}, \bar{y}, \cdot) \bigg\rVert_{L^2(0, y_\eps)} \Bigg],
\end{multline}
Let us define $\hat{Y}_{>0}:= \hat{Y} \cap \{y_N \in \R : y_N > 0 \}$. We square both hand sides of \eqref{ineq: conti psi_eps} and integrate
over $W \times Y$ to get
\begin{multline} \label{proof: ineq Psi 1}
\int_W \int_Y |\hat{\Psi}_\eps(\bar{x},\bar{y}, 0)|^2 d\bar{y} d\bar{x} \leq C (\norma{b}^2_{C^2(Y)} + \norma{\nabla b}^4_{\infty})
\eps^{\alpha-1} \Bigg[ \frac{1}{\eps^4} \norma{D^3_y\hat{u}_\eps}^2_{L^2(W \times \hat{Y}_{>0})} \\
+ \frac{\eps^{2\alpha-2}}{\eps^4} \norma{D^3_y\hat{u}_\eps}^2_{L^2(W \times \hat{Y}_{>0})} + \frac{\eps^{4\alpha - 4}}{\eps^{4}} \norma{D_y^3
\hat{u}_\eps}^2_{L^2(W \times \hat{Y}_{>0})} + \frac{\eps^{2\alpha-4}}{\eps^2} \Big\lVert\frac{\partial^2 \hat{u}_\eps}{\partial y_N^2}
\Big\rVert^2_{L^2(W \times \hat{Y}_{>0})}   \Bigg],
\end{multline}
Due to \eqref{eq: exact int formula} and some basic estimates, \eqref{proof: ineq Psi 1} implies that
\begin{equation}
\label{ineq: estimate psi eps}
\begin{split}
\norma{\hat{\Psi}_\eps(\bar{x},\bar{y}, 0)}^2_{L^2(W \times Y)} &\leq C \norma{D^3 u_\eps}^2_{L^2(\Omega_\eps)} (\eps^\alpha +
\eps^{3\alpha-2} + \eps^{5\alpha - 4}) + C \eps^{3\alpha-4} \Bigg\lVert \frac{\partial^2 u_\eps}{\partial x_N^2}
\Bigg\rVert^2_{L^2(\Omega_\eps \setminus \Omega)} \\
&\leq C (\eps^{\alpha} + \eps^{4\alpha - 4}) + o(\eps^{\alpha})
\end{split}
\end{equation}
where in the last inequality we used that since $\partial^2_{x^2_N} u_\eps$ is in $H^1(\Omega_\eps)$, $\eps > 0$, with uniformly bounded norm, there exists $C > 0$ such that $ \norma{\partial^2_{x^2_N} u_\eps}^2_{L^2(\Omega_\eps \setminus \Omega)} \leq C |\Omega_\eps
\setminus \Omega| \norma{\partial^2_{x_N^2} u_\eps}^2_{W^{1,2}(\Omega_\eps)}$, for all $\eps > 0$.  Note that since $\alpha \geq 3/2 > 1$, \eqref{ineq: estimate psi eps} implies
\begin{equation}
\label{limit: vanishing Psi eps}
\int_W \int_Y \eps^{-1}\Bigg|\hat{\Psi}_\eps(\bar{x},\bar{y}, 0) - \int_Y \hat{\Psi}_\eps(\bar{x},\bar{z}, 0) d\bar{z} \Bigg|^2 d\bar{y}
d\bar{x} = O(\eps^{\alpha-1}) \to 0,
\end{equation}
as $\eps \to 0$. We can rewrite \eqref{limit: vanishing Psi eps} as
\begin{equation}
\label{limit: eps^-1 * sum T_i vanishing}
\int_W \int_Y \big| T_1 + \dots + T_5 \big|^2 d\bar{y} d\bar{x} \to 0, \quad \textup{as $\eps \to 0$,}
\end{equation}
where
\begin{align*}
T_1 &= \frac{1}{\eps^{5/2}} \Bigg( \frac{\partial^2\hat{u}_\eps}{\partial y_i \partial y_j}(\bar{x}, \bar{y},0) - \int_Y
\frac{\partial^2\hat{u}_\eps}{\partial y_i \partial y_j}(\bar{x}, \bar{z},0) d\bar{z}\Bigg);\\
T_2 &= \frac{\eps^{\alpha-1}}{\eps^{5/2}} \Bigg( \frac{\partial^2\hat{u}_\eps}{\partial y_i \partial y_N}(\bar{x}, \bar{y},0) \frac{\partial
b(\bar{y})}{\partial y_j} - \int_Y \frac{\partial^2\hat{u}_\eps}{\partial y_i \partial y_N}(\bar{x}, \bar{z},0)\frac{\partial
b(\bar{z})}{\partial y_j} d\bar{z}\Bigg);\\
T_3 &= \frac{\eps^{\alpha-1}}{\eps^{5/2}} \Bigg( \frac{\partial^2\hat{u}_\eps}{\partial y_j \partial y_N}(\bar{x}, \bar{y},0) \frac{\partial
b(\bar{y})}{\partial y_i} - \int_Y \frac{\partial^2\hat{u}_\eps}{\partial y_j \partial y_N}(\bar{x}, \bar{z},0)\frac{\partial
b(\bar{z})}{\partial y_i} d\bar{z}\Bigg);\\
T_4 &= \frac{\eps^{2\alpha-2}}{\eps^{5/2}} \Bigg( \frac{\partial^2\hat{u}_\eps}{\partial y^2_N}(\bar{x}, \bar{y},0) \frac{\partial
b(\bar{y})}{\partial y_i} \frac{\partial b(\bar{y})}{\partial y_j} - \int_Y \frac{\partial^2\hat{u}_\eps}{\partial y^2_N}(\bar{x},
\bar{z},0)\frac{\partial b(\bar{z})}{\partial y_i} \frac{\partial b(\bar{z})}{\partial y_j} d\bar{z}\Bigg);\\
T_5 &= \frac{\eps^{\alpha -2}}{\eps^{3/2}}  \Bigg( \frac{\partial\hat{u}_\eps}{\partial y_N}(\bar{x}, \bar{y},0) \frac{\partial^2
b(\bar{y})}{\partial y_i \partial y_j} - \int_Y \frac{\partial\hat{u}_\eps}{\partial y_N}(\bar{x}, \bar{z},0)\frac{\partial^2
b(\bar{z})}{\partial y_i \partial y_j} d\bar{z}\Bigg).
\end{align*}
Recall that the function $U_\eps$ defined by
\begin{multline*}
U_\epsilon(\bar{x}, y) = \hat{u}_\epsilon(\bar{x}, y) - \int_Y \Bigg(\hat{u}_\epsilon(\bar{x}, \bar{\zeta}, 0) - \sum_{\abs{\eta}=2} \int_Y
D^\eta_y \hat{u}_\epsilon(\bar{x}, \bar{\zeta}, 0)\, \diff{\bar{\zeta}}\Bigg) \, \frac{\bar{\zeta}^\eta}{\eta!} \, \diff{\bar{\zeta}}\\
- \int_Y \nabla_y \hat{u}_\epsilon(\bar{x}, \bar{\zeta}, 0)\, \diff{\bar{\zeta}} \cdot y - \sum_{\abs{\eta}=2}\int_Y D^\eta_y
\hat{u}_\epsilon(\bar{x}, \bar{\zeta}, 0)\, \diff{\bar{\zeta}}\, \frac{y^\eta}{\eta!},
\end{multline*}
is such that the sequence $(\eps^{-5/2} U_\eps)$ is uniformly bounded in $L^2(W, H^3(Y \times (d,0))$, for any $d< 0$, see Lemma
\ref{lemma: unfolding convergence poly}. Note also that $D^\eta_y U_\eps = D^\eta_y \hat{u_\epsilon} - \int_Y D^\eta_y \hat{u_\epsilon}(\cdot,
\bar{z}, \cdot) d\bar{z}$ for any $|\eta| = 2$. Using these facts we deduce that
\[
\begin{split}
\int_W \int_Y |T_1|^2 d\bar{y} d\bar{x} &= \int_W \int_Y \Bigg|\eps^{-5/2}\frac{ \partial^2 U_\eps}{\partial y_i \partial y_j}(\bar{x},
\bar{y}, 0)\Bigg|^2 d\bar{y} d\bar{x}\\
&\leq C \Big\lVert\eps^{-5/2}\frac{\partial^2 U_\eps}{\partial y_i \partial y_j}\Big\rVert^2_{L^2(W, H^1(Y \times (-1,0))}\\
&\leq C \lVert\eps^{-5/2} D_y^3 \hat{u}_\eps \rVert^2_{L^2(W \times Y \times (-1,0))} \leq C \norma{D^3 u_\eps}^2_{L^2(\Omega)},
\end{split}
\]
where we have used a trace inequality, the Poincaré-Wirtinger inequality, and the exact integration formula \eqref{eq: exact int formula}.
Hence $T_1$ is bounded in $L^2(W \times Y)$, uniformly in $\eps > 0$.

Consider now $T_2$. Note that the function $\frac{\partial b}{\partial y_j}$ has null average over $Y$ because of
periodicity. Hence,
\[
\int_Y \frac{\partial^2\hat{u}_\eps}{\partial y_i \partial y_N}(\bar{x}, \bar{z},0)\frac{\partial b(\bar{z})}{\partial y_j} d\bar{z} =  \int_Y
\frac{\partial b(\bar{z})}{\partial y_j} \Bigg( \frac{\partial^2\hat{u}_\eps}{\partial y_i \partial y_N}(\bar{x}, \bar{z},0) - \int_Y
\frac{\partial^2\hat{u}_\eps}{\partial y_i \partial y_N}(\bar{x}, \bar{t},0) d\bar{t} \Bigg) d\bar{z}
\]
and
{\small
\begin{equation}
\label{ineq: average decay}
\begin{split}
&\int_W \int_Y \eps^{2\alpha-2-5} \Bigg|\int_Y \frac{\partial b(\bar{z})}{\partial y_j} \Bigg( \frac{\partial^2\hat{u}_\eps}{\partial y_i
\partial y_N}(\bar{x}, \bar{z},0) - \int_Y \frac{\partial^2\hat{u}_\eps}{\partial y_i \partial y_N}(\bar{x}, \bar{t},0) d\bar{t} \Bigg)
d\bar{z}\Bigg|^2 d\bar{y} d\bar{x}\\
&\leq C\frac{\eps^{2\alpha}}{\eps^2} \int_W \int_Y\int_Y \Bigg|\eps^{-5/2}\Bigg( \frac{\partial^2\hat{u}_\eps}{\partial y_i \partial y_N}(\bar{x},
\bar{z},0) - \int_Y \frac{\partial^2\hat{u}_\eps}{\partial y_i \partial y_N}(\bar{x}, \bar{t},0) d\bar{t} \Bigg)\Bigg|^2 d\bar{z} d\bar{y}
d\bar{x}\\
&\leq C \eps^{2\alpha-2} \norma{\eps^{-5/2} \partial^2_{y_i y_N} U_\eps(\cdot,\cdot, 0)}^2_{L^2(W \times Y)} \\
&\leq C \eps^{2\alpha-2} \norma{\eps^{-5/2}D_y^3\hat{u}_\eps}^2_{L^2(W \times Y \times(-1,0))} \to 0,
\end{split}
\end{equation}
}

as $\eps \to 0$, for all $\alpha>1$. We deduce that
\begin{equation}
\label{ineq: estimate T2}
\begin{split}
\int_W\int_Y |T_2|^2 d\bar{y} d\bar{x} &\leq C \int_W \int_Y \Bigg|\frac{\eps^{\alpha-1}}{\eps^{5/2}} \Bigg(
\frac{\partial^2\hat{u}_\eps}{\partial y_i \partial y_N}(\bar{x}, \bar{y},0) \frac{\partial b(\bar{y})}{\partial y_j} \Bigg) \Bigg|^2 d\bar{y}
d\bar{x}\\
&+C \int_W \int_Y \Bigg|\int_Y \eps^{\alpha-1-5/2} \frac{\partial^2\hat{u}_\eps}{\partial y_i \partial y_N}(\bar{x}, \bar{z},0) \frac{\partial
b(\bar{z})}{\partial y_j} d\bar{z}\Bigg|^2 d\bar{y} d\bar{x}\\
&\leq C \int_W \int_Y \Bigg|\frac{\eps^{\alpha}}{\eps^{3/2}} \Bigg( \frac{1}{\eps^2}\frac{\partial^2\hat{u}_\eps}{\partial y_i \partial
y_N}(\bar{x}, \bar{y},0) \frac{\partial b(\bar{y})}{\partial y_j} \Bigg) \Bigg|^2 d\bar{y} d\bar{x} + o(1),
\end{split}
\end{equation}
as $\eps \to 0$. We claim that
\begin{equation}
\label{claim: convergence unfolding}
\frac{1}{\eps^2}\frac{\partial^2\hat{u}_\eps}{\partial y_i \partial y_N}(\bar{x}, \bar{y},0) \frac{\partial b(\bar{y})}{\partial y_j} \to
\frac{\partial^2 u}{\partial x_i \partial x_N}(\bar{x}, 0) \frac{\partial b(\bar{y})}{\partial y_j},
\end{equation}
in $L^2(W \times Y)$ as $\eps \to 0$. Since $u_\eps|_{\Omega} \to u$ weakly in $H^3(\Omega)$, by the compactness of the trace operator we have that
\begin{equation}
\label{limit: trace convergence}
\frac{\partial^2 u_\eps}{\partial x_i \partial x_N}(\bar{x}, 0) \to \frac{\partial^2 u}{\partial x_i \partial x_N}(\bar{x}, 0),
\end{equation}
in $L^2(W)$, as $\eps \to 0$. Now define
\[
\overline{\frac{\partial^2 u_\eps}{\partial x_i \partial x_N}}(\bar{x}) := \frac{1}{\eps^{N-1}} \int_{C_\eps(\bar{x})} \frac{\partial^2
u_\eps}{\partial x_i \partial x_N}(\bar{t},0)\, d\bar{t},
\]
where $C_\eps(\bar{x})$ is as in \eqref{def: anistropic unfolding cell}. Note that, by a change of variable,
\[
\overline{\frac{\partial^2 u_\eps}{\partial x_i \partial x_N}}(\bar{x})  = \int_Y \widehat{\frac{\partial^2 u_\eps}{\partial x_i \partial
x_N}}(\bar{x},\bar{z},0) d\bar{z} = \frac{1}{\eps^2} \int_Y \frac{\partial^2 \hat{u}_\eps}{\partial y_i \partial y_N} (\bar{x},\bar{z},0)\,
d\bar{z}.
\]
By \eqref{limit: trace convergence} we deduce that
\[
\overline{\frac{\partial^2 u_\eps}{\partial x_i \partial x_N}} \to \frac{\partial^2 u}{\partial x_i \partial x_N}(\cdot,0),
\]
strongly in $L^2(W)$ as $\eps \to 0$. Here, we have used the fact that if a sequence of functions $v_\eps$ converges strongly in $L^2$ to $v$ then $\overline{v_\eps}$ converges strongly in $L^2$ to $v$. We give a proof of this in Lemma \ref{lemma: easy averages} in Appendix (B). Since $\eps^{-5/2}\partial_{y_i y_N}U_\eps$ is uniformly bounded in $L^2(W \times Y)$, for all $\eps
>0$ due to Lemma \ref{lemma: exact int formula}, it follows that
\[
\frac{1}{\eps^2} \bigg(\frac{\partial^2 \hat{u}_\eps}{\partial y_i \partial y_N}(\cdot,\cdot,0) - \int_Y \frac{\partial^2
\hat{u}_\eps}{\partial y_i \partial y_N}(\cdot, \bar{z},0) d\bar{z} \bigg) \to 0,
\]
in $L^2(W \times Y)$ as $\eps \to 0$. Hence, $\frac{1}{\eps^2}\frac{\partial^2\hat{u}_\eps}{\partial y_i \partial y_N}(\bar{x}, \bar{y},0)  \to \frac{\partial^2 u}{\partial x_i \partial
x_N}(\bar{x}, 0)$ in $L^2(W \times Y)$ as $\eps \to 0$, which proves the claim. Since $\alpha > 3/2$, by recalling \eqref{ineq: estimate T2} we then deduce that $T_2$ vanishes in $L^2(W \times Y)$ as $\eps \to 0$.

$T_3$ is exactly $T_2$ with swapped indexes $i$ and $j$, hence also $T_3$ vanishes in $L^2(W \times Y)$ as $\eps
\to 0$.
%

We then consider $T_4$. By arguing as in \eqref{claim: convergence unfolding} we deduce that
\begin{equation}
\label{proof: limit 1}
\frac{1}{\eps^{2}} \frac{\partial^2\hat{u}_\eps}{\partial y^2_N}(\bar{x}, \bar{y},0) \frac{\partial b(\bar{y})}{\partial y_i} \frac{\partial
b(\bar{y})}{\partial y_j} \to \frac{\partial^2 u}{\partial y^2_N}(\bar{x}, 0) \frac{\partial b(\bar{y})}{\partial y_i} \frac{\partial
b(\bar{y})}{\partial y_j},
\end{equation}
in $L^2(W \times Y)$ as $\eps \to 0$, so the integral in $Y$ of the left-hand side of \eqref{proof: limit 1} is convergent.
Thus,
\begin{equation}
\label{limit: vanishing T4}
T_4 = \frac{\eps^{2\alpha}}{\eps^{5/2}} \Bigg( \frac{1}{\eps^2}\frac{\partial^2\hat{u}_\eps}{\partial y^2_N}(\bar{x}, \bar{y},0)
\frac{\partial b(\bar{y})}{\partial y_i} \frac{\partial b(\bar{y})}{\partial y_j} - \int_Y
\frac{1}{\eps^2}\frac{\partial^2\hat{u}_\eps}{\partial y^2_N}(\bar{x}, \bar{z},0)\frac{\partial b(\bar{z})}{\partial y_i} \frac{\partial
b(\bar{z})}{\partial y_j} d\bar{z}\Bigg) \to 0,
\end{equation}
in $L^2(W \times Y)$ as $\eps \to 0$ for all $\alpha > 5/4$, hence in particular for any $\alpha \geq 3/2$.\\
Finally, we consider $T_5$. Arguing as in the proof of Claim \eqref{claim: convergence unfolding} we can prove that
\begin{equation}
\label{claim: convergence T5}
\frac{1}{\eps} \frac{\partial \hat{u}_\eps}{\partial y_N}(\bar{x}, \bar{y}, 0) \frac{\partial^2 b(\bar{y})}{\partial y_i \partial y_j} \to
\frac{\partial u}{\partial y_N}(\bar{x}, 0) \frac{\partial^2 b(\bar{y})}{\partial y_i \partial y_j},
\end{equation}
in $L^2(W \times Y)$ as $\eps \to 0$ and
\begin{equation}
\label{claim: convergence T5 2}
\int_Y \frac{1}{\eps} \frac{\partial \hat{u}_\eps}{\partial y_N}(\bar{x}, \bar{z}, 0) \frac{\partial^2 b(\bar{z})}{\partial y_i \partial
y_j}d\bar{z} \to \frac{\partial u}{\partial y_N}(\bar{x}, 0) \int_Y \frac{\partial^2 b(\bar{z})}{\partial y_i \partial y_j} d\bar{z} = 0,
\end{equation}
in $L^2(W \times Y)$ as $\eps \to 0$, where the right-hand side of \eqref{claim: convergence T5 2} is zero due to periodicity of $b$.
We now consider different cases according to the value of $\alpha$. \\

\vspace{0.1cm}

\noindent\textbf{Case $3/2 < \alpha < 5/2$.} In this case, by summarising the previous results we have that $T_1$ is uniformly bounded in $L^2(W \times
Y)$ as $\eps \to 0$, whereas $T_2, T_3, T_4$ tend to zero in $L^2(W \times Y)$ as $\eps \to 0$. Then \eqref{limit: eps^-1 * sum T_i vanishing} implies that there exists a constant $M > 0$ such that
\[
\Bigg(\int_W\int_Y |T_5|^2 d\bar{y} d\bar{x}\Bigg)^{1/2} \leq \Bigg(\int_W\int_Y |T_1 + T_2 + T_3 + T_4|^2 d\bar{y} d\bar{x}\Bigg)^{1/2} +
o(1) \leq M,
\]
as $\eps \to 0$. Thus,
\[
\Bigg\lVert \frac{1}{\eps}\frac{\partial \hat{u}_\eps}{\partial y_N}(\bar{x}, \bar{y},0) \frac{\partial^2 b(\bar{y})}{\partial y_i \partial
y_j} - \int_Y \frac{1}{\eps}\frac{\partial \hat{u}_\eps}{\partial y_N}(\bar{x}, \bar{z},0) \frac{\partial^2 b(\bar{z})}{\partial y_i \partial
y_j} d\bar{z} \Bigg\rVert_{L^2(W \times Y)} = O(\eps^{5/2-\alpha}),
\]
as $\eps \to 0$. By letting $\eps \to 0$ and recalling \eqref{claim: convergence T5} and \eqref{claim: convergence T5 2} we deduce that
$\frac{\partial u}{\partial y_N}(\bar{x}, 0) \frac{\partial^2 b(\bar{y})}{\partial y_i \partial y_j} = 0$, for a.e. $\bar{x} \in W$, for a.e. $\bar{y} \in Y$, and since $b$ is not affine we deduce that
\begin{equation}
\label{case 3/2<alpha<5/2}
\frac{\partial u}{\partial x_N}(\bar{x}, 0) = 0,
\end{equation}
for a.e. $\bar{x} \in W$. We conclude that $u \in H^3(\Omega) \cap H^2_{0,*}$.

\vspace{0.3cm}

\noindent\textbf{Case $\alpha = 5/2$.} In this case, we have the estimate
\[
\Bigg(\int_W\int_Y |T_1 + T_5|^2 d\bar{y} d\bar{x}\Bigg)^{1/2} \leq \Bigg(\int_W\int_Y |T_2 + T_3 + T_4|^2 d\bar{y} d\bar{x}\Bigg)^{1/2} +
o(1) = o(1),
\]
as $\eps \to 0$. Thus,
\[
\frac{1}{\eps^{5/2}} \Bigg( \frac{\partial^2\hat{u}_\eps}{\partial y_i \partial y_j}(\bar{x}, \bar{y},0) - \int_Y
\frac{\partial^2\hat{u}_\eps}{\partial y_i \partial y_j}(\bar{x}, \bar{z},0) d\bar{z}\Bigg) +
\frac{1}{\eps}\frac{\partial\hat{u}_\eps}{\partial y_N}(\bar{x}, \bar{y},0) \frac{\partial^2 b(\bar{y})}{\partial y_i \partial y_j} \to 0,
\]
as $\eps \to 0$. Now since $(\eps^{-5/2} U_\eps)$ is uniformly bounded in $L^2(W \times Y \times (d,0))$, there exists a subsequence of
$(\eps^{-5/2} U_\eps)$ and a function $\hat{u} \in L^2(W, H^3(Y \times (d,0)))$ such that
\begin{equation}
\label{limit: U_eps convergence}
\eps^{-5/2} U_\eps \rightharpoonup \hat{u},
\end{equation}
in $L^2(W, H^3(Y \times (d,0)))$. \eqref{limit: U_eps convergence} implies that
\[
\frac{1}{\eps^{5/2}} \Bigg( \frac{\partial^2\hat{u}_\eps}{\partial y_i \partial y_j}(\bar{x}, \bar{y},0) - \int_Y
\frac{\partial^2\hat{u}_\eps}{\partial y_i \partial y_j}(\bar{x}, \bar{z},0) d\bar{z}\Bigg) \to \frac{\partial^2\hat{u}}{\partial y_i \partial
y_j}(\bar{x}, \bar{y},0),
\]
strongly in $L^2(W \times Y)$ as $\eps \to 0$. Moreover, according to \eqref{claim: convergence T5} we deduce that
\begin{equation}
\label{case alpha=5/2}
\frac{\partial^2\hat{u}}{\partial y_i \partial y_j}(\bar{x}, \bar{y},0) = -\frac{\partial u}{\partial y_N}(\bar{x}, 0) \frac{\partial^2
b(\bar{y})}{\partial y_i \partial y_j},
\end{equation}
for a.e. $\bar{x} \in W$, a.e. $\bar{y} \in Y$, which is \eqref{eq: limit case hatu eq}. \\

\vspace{0.3cm}

\vspace{0.3cm}

\noindent\textbf{Case $\alpha \leq 1$.} In this case we give a more direct proof based on a different definition of the unfolding operator. We define
\begin{equation}
\label{proof: def hat Y tri}
\hat{Y}= \{ (\bar{y}, y_N) : \bar{y} \in Y, -1 < y_N < b(\bar{y})\},
\end{equation}
and
\begin{equation}
\label{proof: def hat u eps tri}
\hat{u}_\eps(\bar{x}, \bar{y}, y_N) := u_\eps\bigg(\eps \Big[\frac{\bar{x}}{\eps}\Big] + \eps \bar{y}, \eps^\alpha y_N\bigg),
\end{equation}
for all $(\bar{x}, y) \in W \times \hat{Y}$, for all $u_\eps \in H^3(\Omega_\eps)$. Note that $\hat{u}_\eps$, $\eps \in (0,1]$, are defined on a fixed domain of $\R^N$. Then, starting from the identity \eqref{eq: boundary equality} we deduce the analogous of \eqref{eq: boundary equality 2}, which namely reads
\begin{equation}
\label{eq: unfolding equality alfa < 1}
\begin{split}
&\frac{1}{\eps^2}\frac{\partial^2 \hat{u}_\eps}{\partial y_i \partial y_j}(\bar{x}, \bar{y},  b(\bar{y})) +
\frac{\eps^{\alpha-1}}{\eps^{\alpha+1}}\frac{\partial^2 \hat{u}_\eps}{\partial y_i\partial y_N}(\bar{x}, \bar{y}, b(\bar{y})) \frac{\partial
b(\bar{y})}{\partial y_j}\\
&+ \frac{\eps^{\alpha-1}}{\eps^{\alpha + 1}}\frac{\partial^2 \hat{u}_\eps}{\partial y_j\partial y_N}(\bar{x}, \bar{y}, b(\bar{y}))
\frac{\partial b(\bar{y})}{\partial y_i} + \frac{\eps^{2\alpha-2}}{\eps^{2\alpha}}\frac{\partial^2 \hat{u}_\eps}{\partial y^2_N}(\bar{x},
\bar{y}, b(\bar{y})) \frac{\partial b(\bar{y})}{\partial y_i}\frac{\partial b(\bar{y})}{\partial y_j}\\
&+ \frac{\eps^{\alpha-2}}{\eps^{\alpha}} \frac{\partial \hat{u}_\eps}{\partial y_N}(\bar{x}, \bar{y}, b(\bar{y}))\frac{\partial^2
b(\bar{y})}{\partial y_i \partial y_j} = 0.
\end{split}
\end{equation}
If $\alpha = 1$, by arguing as in \eqref{proof: conv hat u eps to u} below, we have
\[
\frac{1}{\eps^2}\frac{\partial^2 \hat{u}_\eps}{\partial y_i \partial y_j}(\bar{x}, \bar{y},  b(\bar{y})) \to \frac{\partial^2 u}{\partial x_i
\partial x_j}(\bar{x},0),
\]
\[
\frac{1}{\eps^{2}}\frac{\partial^2 \hat{u}_\eps}{\partial y_i\partial y_N}(\bar{x}, \bar{y}, b(\bar{y})) \frac{\partial b(\bar{y})}{\partial
y_j} \to \frac{\partial^2 u}{\partial x_i \partial x_N}(\bar{x},0) \frac{\partial b(\bar{y})}{\partial y_j},
\]
\[
\frac{1}{\eps^{2}}\frac{\partial^2 \hat{u}_\eps}{\partial y^2_N}(\bar{x}, \bar{y}, b(\bar{y})) \frac{\partial b(\bar{y})}{\partial
y_i}\frac{\partial b(\bar{y})}{\partial y_j} \to \frac{\partial^2 u}{\partial x_N^2}(\bar{x},0) \frac{\partial b(\bar{y})}{\partial
y_i}\frac{\partial b(\bar{y})}{\partial y_j},
\]
as $\eps \to 0$, where the limits are taken in $L^2(W \times Y)$. According to \eqref{eq: unfolding equality alfa < 1}, we immediately
discover that
\begin{equation}
\label{ineq: decay normal derivative}
\Bigg \lVert \frac{1}{\eps} \frac{\partial \hat{u}_\eps}{\partial y_N}(\bar{x}, \bar{y}, b(\bar{y})) \Bigg\rVert_{L^2(W \times \hat{Y})} \leq
C\eps,
\end{equation}
for all $\eps>0$. By \eqref{ineq: decay normal derivative} we deduce that
\begin{equation}
\label{partial u partial x_N = 0 alpha =1}
\frac{\partial u}{\partial x_N} (\bar{x},0) = 0,
\end{equation}
and that there exists a function $\zeta \in L^2(W)$ such that, up to a subsequence,
\[
\frac{1}{\eps^2}\frac{\partial \hat{u}_\eps}{\partial y_N}(\bar{x}, \bar{y}, b(\bar{y})) \rightharpoonup \zeta(\bar{x}),
\]
in $L^2(W \times Y)$ as $\eps \to 0$. The fact that $\zeta$ does not depend on $\bar{y}$ is an easy consequence of the following argument. Let
$\varphi \in C^\infty_c(W \times Y)$. Then
\[
\int_{W \times Y} \frac{1}{\eps^2}\frac{\partial \hat{u}_\eps}{\partial y_N}(\bar{x}, \bar{y}, b(\bar{y})) \frac{\partial \varphi}{\partial
y_i} \, d\bar{x} d\bar{y} = - \int_{W \times Y} \frac{1}{\eps^2}\frac{\partial^2 \hat{u}_\eps}{\partial y_N \partial y_i}(\bar{x}, \bar{y},
b(\bar{y})) \varphi \, d\bar{x} d\bar{y},
\]
and passing to the limit as $\eps \to 0$ we deduce that
\begin{equation}
\label{eq: weak derivative zeta}
\int_{W \times Y} \zeta \frac{\partial \varphi}{\partial y_i} \,d\bar{x} d\bar{y} = - \int_{W \times Y} \frac{\partial^2 u}{\partial x_N
\partial x_i}(\bar{x}, 0) \varphi \, d\bar{x} d\bar{y} = 0,
\end{equation}
where we have used that $\frac{\partial^2 u}{\partial x_N \partial x_i}(\bar{x}, 0) = 0$ because of \eqref{partial u partial x_N = 0 alpha
=1}. Equation\eqref{eq: weak derivative zeta} implies that $\zeta$ is weakly differentiable in $y_i$ and that $\frac{\partial \zeta}{\partial y_i} = 0$.\\
Taking the limit as $\eps \to 0$ in $L^2(W \times Y)$ in \eqref{eq: unfolding equality alfa < 1} we deduce that
\begin{equation}
\label{eq: unfolding equality limit}
\begin{split}
&\frac{\partial^2 u}{\partial x_i \partial x_j}(\bar{x}, 0) + \frac{\partial^2 u}{\partial x_i \partial x_N}(\bar{x}, 0) \frac{\partial
b(\bar{y})}{\partial y_j} + \frac{\partial^2 u}{\partial x_j \partial x_N}(\bar{x}, 0) \frac{\partial b(\bar{y})}{\partial y_i}\\
&+ \frac{\partial^2 u}{\partial x^2_N}(\bar{x}, 0) \frac{\partial b(\bar{y})}{\partial y_i}\frac{\partial b(\bar{y})}{\partial y_j} +
\zeta(\bar{x})\frac{\partial^2 b(\bar{y})}{\partial y_i \partial y_j} = 0.
\end{split}
\end{equation}
Because of \eqref{partial u partial x_N = 0 alpha =1} the first three summands in \eqref{eq: unfolding equality limit} are zero. Hence, \eqref{eq: unfolding equality limit} implies that
\begin{equation} \label{proof: identity der alpha < 1}
\frac{\partial^2 u}{\partial x^2_N}(\bar{x}, 0) \frac{\partial b(\bar{y})}{\partial y_i}\frac{\partial b(\bar{y})}{\partial y_j} +
\zeta(\bar{x})\frac{\partial^2 b(\bar{y})}{\partial y_i \partial y_j} = 0.
\end{equation}
Recall now that since $b$ is $Y$-periodic, its derivatives are periodic and with null average on $Y$. An integration in $Y$ in \eqref{proof: identity der alpha < 1} yields
\[
\frac{\partial^2 u}{\partial x^2_N}(\bar{x}, 0) \int_Y \frac{\partial b(\bar{y})}{\partial y_i}\frac{\partial b(\bar{y})}{\partial y_j}
d\bar{y} = 0,
\]
for almost all $\bar{x} \in W$. Since this holds for all $i,j=1, \dots, N-1$ we can in particular choose $i=j$ so that $
\frac{\partial^2 u}{\partial x^2_N}(\bar{x}, 0) \int_Y  \lvert \nabla b  \rvert^2 d\bar{y} =  0$, and since $b$ is non constant it must be
$\frac{\partial^2 u}{\partial x^2_N}(\bar{x}, 0) = 0$ for almost all $\bar{x} \in W$.\\

If $\alpha < 1$ we can argue in a similar way. Namely, we multiply each side of \eqref{eq: unfolding equality alfa < 1} by $\eps^{2-2\alpha}$ in order to obtain
\begin{equation}
\label{eq: unfolding equality alfa < 1 multiplied}
\begin{split}
&\frac{1}{\eps^{2\alpha}}\frac{\partial^2 \hat{u}_\eps}{\partial y_i \partial y_j}(\bar{x}, \bar{y},  b(\bar{y})) +
\frac{\eps^{1-\alpha}}{\eps^{\alpha+1}}\frac{\partial^2 \hat{u}_\eps}{\partial y_i\partial y_N}(\bar{x}, \bar{y}, b(\bar{y})) \frac{\partial
b(\bar{y})}{\partial y_j}\\
&+ \frac{\eps^{1-\alpha}}{\eps^{\alpha + 1}}\frac{\partial^2 \hat{u}_\eps}{\partial y_j\partial y_N}(\bar{x}, \bar{y}, b(\bar{y}))
\frac{\partial b(\bar{y})}{\partial y_i} + \frac{1}{\eps^{2\alpha}}\frac{\partial^2 \hat{u}_\eps}{\partial y^2_N}(\bar{x}, \bar{y},
b(\bar{y})) \frac{\partial b(\bar{y})}{\partial y_i}\frac{\partial b(\bar{y})}{\partial y_j}\\
&+ \frac{1}{\eps^{2\alpha}} \frac{\partial \hat{u}_\eps}{\partial y_N}(\bar{x}, \bar{y}, b(\bar{y}))\frac{\partial^2 b(\bar{y})}{\partial y_i
\partial y_j} = 0.
\end{split}
\end{equation}
Since $u(\bar{x},0)=0$, a.a $x \in W$, the first three summands in \eqref{eq: unfolding equality alfa < 1 multiplied} are vanishing as $\eps \to 0$. Then we deduce that
\[
\frac{\partial^2 u}{\partial x^2_N}(\bar{x},0) \frac{\partial b(\bar{y})}{\partial y_i}\frac{\partial b(\bar{y})}{\partial y_j} + \lim_{\eps
\to 0}\frac{1}{\eps^{2\alpha}} \frac{\partial \hat{u}_\eps}{\partial y_N}(\bar{x}, \bar{y}, b(\bar{y}))\frac{\partial^2 b(\bar{y})}{\partial
y_i \partial y_j} = 0.
\]
This first implies that
\[
\Bigg \lVert  \frac{1}{\eps^\alpha} \frac{\partial \hat{u}_\eps}{\partial y_N}(\bar{x},\bar{y}, b(\bar{y})) \frac{\partial^2 b}{\partial y_i
\partial y_j}  \Bigg \rVert_{L^2(W \times Y)} \leq C \eps^\alpha,
\]
hence $\frac{\partial u}{\partial x_N}(\bar{x},0) = 0$. Moreover, we deduce that up to a subsequence there exists $\zeta \in L^2(W)$ such that $\frac{1}{\eps^{2\alpha}}\frac{\partial \hat{u}_\eps}{\partial y_N}(\bar{x}, \bar{y}, b(\bar{y})) \rightharpoonup \zeta(\bar{x})$
in $L^2(W \times Y)$ as $\eps \to 0$. Then arguing as in the case $\alpha = 1$ we deduce that $\frac{\partial^2 u}{\partial x_N^2}(\bar{x},0) = 0$.
\end{proof}

\begin{proof}[Proof of Theorem \ref{thm: spectral conv tri weak}(iii),(iv)]
We first prove Claim $(iii)$. We will show that the Condition (C), defined in \cite[Def. 3.1]{ArrLamb} holds with $V(\Omega_\eps) =
H^3(\Omega_\eps) \cap H^1_0(\Omega_\eps)$ and $V(\Omega) = H^3(\Omega) \cap H^1_0(\Omega) \cap H^2_{0,*}(\Omega)$. In  \cite[Def. 3.1]{ArrLamb} we choose $$K_\eps = \{ x \in \Omega : x_N < -\eps \},$$ $\eps \in (0,1]$, $T_\eps : V(\Omega) \to V(\Omega_\eps)$ as in \eqref{def: T_eps} and $E_\eps: V(\Omega_\eps) \to H^m(\Omega)$ as the restriction operator $E_\eps u_\eps = u_\eps|_{\Omega}$, $\eps \in (0,1]$.  With this choices it is not difficult to verify that conditions (C1), (C2)(i), (C2)(iii), (C3)(i) and (C3)(ii) hold true. Then it is sufficient to prove the validity of conditions $(C2)(ii)$ and $(C3)(iii)$.

In order to show that $(C2)(ii)$ holds it is sufficient to use Lemma~\ref{lemma: degeneration H^3 H^1_0}(iii) and its proof. Indeed, if $\alpha > 3/2$ then $\lim_{\eps \to 0} \norma{T_\eps \varphi}_{H^3(\Omega_\eps \setminus K_\eps)} = 0$ for all $\varphi \in V(\Omega)$.

Condition $(C3)(iii)$ now follows directly from Lemma \ref{lemma: degeneration H^3 H^1_0}(iii), since we have proved that if $u_\eps \in V(\Omega_\eps)$ is such that $u_\eps|_{\Omega} \rightharpoonup u$ and $3/2 < \alpha < 5/2$, then $u \in V(\Omega)$.

Hence Condition (C) holds and \cite[Thm 3.5]{ArrLamb} now yields the claim.\\
The proof of Claim $(iv)$ is similar. We show that Condition (C) holds with $V(\Omega_\eps) = H^3(\Omega_\eps) \cap H^1_0(\Omega_\eps)$, $\eps \in (0,1]$, $V(\Omega) = H^3(\Omega) \cap H^1_0(\Omega) \cap H^3_{0,*}(\Omega)$, $T_\eps$ the extension-by-zero operator, and $E_\eps$ the restriction operator defined above. Then conditions (C1)-(C3) hold true. Note that Condition (C3)(iii) follows directly from Lemma
\ref{lemma: degeneration H^3 H^1_0}(iv).\\
\end{proof}

\section{Proof of Theorem \ref{thm: spectral conv tri weak}(ii)}
\label{sec: homog}
In this section, we shall consider the case $\alpha = 5/2$ of Theorem \ref{thm: spectral conv tri weak}. We refer to Section \ref{sec: auxil}
for the notation about $\Phi_\epsilon$, $h_\epsilon$, $T_\epsilon$, $C^k_\epsilon$,  $\hat{u}$, $w^{3,2}_{Per_Y}(Y \times (-\infty, 0))$. We
divide the proof in two subsections. Since the proof follows the same strategy as \cite{FerLamb}, \cite{ArrLamb}, we will only sketch the proofs and refer to \cite{ArrLamb} for further details in the case of the biharmonic operator with SBC.

\subsection{Macroscopic limit.}
Let $f_\epsilon\in L^2(\Omega_\epsilon)$ and $f\in L^2(\Omega)$ be such that $f_\epsilon \rightharpoonup f$ in $L^2(\R^N)$ as $\epsilon \to
0$, with the understanding that the functions are extended by zero outside their natural domains. Let $v_\epsilon \in V(\Omega_\epsilon) =
H^3(\Omega_\epsilon) \cap H^1_0(\Omega_\epsilon)$ be such that
\begin{equation}
\label{eq: Poisson prblm weak BC}
A_{\Omega_\epsilon} v_\epsilon = f_\epsilon,
\end{equation}
for all $\epsilon >0$ small enough. Then $\norma{v_\epsilon}_{H^3(\Omega_\epsilon)} \leq M$ for all $\epsilon > 0$ sufficiently small,
hence, possibly passing to a subsequence there exists $v \in H^3(\Omega)\cap H^1_0(\Omega)$ such that $v_\epsilon \rightharpoonup v$
in $H^3(\Omega)$ and $v_\epsilon \rightarrow v$ in $L^2(\Omega)$.\\
Let $\varphi \in V(\Omega) = H^3(\Omega) \cap H^1_0(\Omega)$ be a fixed test function. Since $T_\epsilon \varphi \in
V(\Omega_\epsilon)$, by \eqref{eq: Poisson prblm weak BC} we get
\begin{equation}
\label{eq: Poisson 1 weak BC}
\int_{\Omega_\epsilon} D^3v_\epsilon : D^3 (T_\epsilon \varphi) \, \diff{x} + \int_{\Omega_\epsilon}v_\epsilon T_\epsilon \varphi\, \diff{x} =
\int_{\Omega_\epsilon} f_\epsilon T_\epsilon \varphi\, \diff{x},
\end{equation}
and passing to the limit as $\epsilon \to 0$ we have that
\[
\int_{\Omega_\epsilon} v_\epsilon T_\epsilon \varphi\, \diff{x} \to \int_{\Omega} v \varphi \,\diff{x}, \quad \int_{\Omega_\epsilon}
f_\epsilon T_\epsilon \varphi\, \diff{x} \to \int_{\Omega} f \varphi \,\diff{x}.
\]

Now consider the first integral in the right-hand side of \eqref{eq: Poisson 1 weak BC}. Let us define $K_\epsilon = W \times (-1,
-\epsilon)$. By splitting the integral in three terms corresponding to $\Omega_\epsilon \setminus \Omega$, $\Omega \setminus K_\epsilon$ and
$K_\epsilon$ and by arguing as in \cite[Section 8.3]{ArrLamb} one can show that
\[
\int_{K_\epsilon} D^3v_\epsilon : D^3 (T_\epsilon \varphi) \, \diff{x} \to \int_\Omega D^3 v : D^3 \varphi \, \diff{x}, \quad
\int_{\Omega_\epsilon \setminus \Omega} D^3 v_\epsilon : D^3 (T_\epsilon \varphi) \, \diff{x} \to 0,
\]
as $\epsilon \to 0$. Let $Q_\eps = \widehat{W}_\eps \times (-\eps, 0)$. We split again the remaining integral in two summands,
\begin{equation}\label{eq: integralsplit weak BC}
\int_{\Omega_\epsilon \setminus K_\epsilon} D^3v_\epsilon : D^3 (T_\epsilon \varphi) \, \diff{x} = \int_{\Omega_\epsilon \setminus (K_\epsilon
\cup Q_\epsilon)} D^3v_\epsilon : D^3 (T_\epsilon \varphi) \, \diff{x} + \int_{Q_\epsilon} D^3v_\epsilon : D^3 (T_\epsilon \varphi) \, \diff{x}.
\end{equation}
Again, by arguing as in \cite[Section 8.3]{ArrLamb} it is possible to prove that
\[
\int_{\Omega_\epsilon \setminus (K_\epsilon \cup Q_\epsilon)} D^3v_\epsilon : D^3 (T_\epsilon \varphi) \, \diff{x} \to 0,
\]
as $\epsilon \to 0$. We now require two technical lemmata.

\begin{lemma}
\label{lemma: convergence b weak}
For all $y \in Y \times (-1,0)$ and $i,j,k = 1,\dots,N$ the functions $\hat{h}_\epsilon(\bar{x},y)$, $\widehat{\frac{\partial
h_\epsilon}{\partial x_i}}(\bar{x},y)$, $\widehat{\frac{\partial^2 h_\epsilon}{\partial x_i\partial x_j}}(\bar{x},y)$ and
$\widehat{\frac{\partial^3 h_\epsilon}{\partial x_i\partial x_j \partial x_k}}(\bar{x},y)$ are independent of $\bar{x}$. Moreover,
\[
\hat{h}_\epsilon(\bar{x},y) = O(\epsilon^{5/2}), \:\: \widehat{\frac{\partial h_\epsilon}{\partial x_i}}(\bar{x},y) = O(\epsilon^{3/2}), \:\: \widehat{\frac{\partial^2 h_\epsilon}{\partial x_i \partial x_j}}(\bar{x},y) = O(\epsilon^{1/2}) ,
\]
as $\epsilon \to 0$, for all $i,j = 1,\dots,N$, uniformly in $y \in Y \times (-1,0)$, and
\[
\epsilon^{1/2} \widehat{\frac{\partial^3 h_\epsilon}{\partial x_i\partial x_j \partial x_k}} (\bar{x},y) \to \frac{\partial^3 (b(\bar{y})(y_N
+ 1)^4)}{\partial y_i \partial y_j \partial y_k},
\]
as $\epsilon \to 0$, for all $i,j,k = 1,\dots,N$, uniformly in $y \in Y \times (-1,0)$.
\end{lemma}
\begin{proof}
We refer to \cite[Lemma 4]{FerLamb} and \cite[Lemma 8.27]{ArrLamb}, where similar computations were carried out in the case of strong intermediate boundary conditions.
\end{proof}

\begin{lemma}
Let $v_\epsilon \in V(\Omega_\epsilon) = H^3(\Omega_\epsilon) \cap H^1_0(\Omega_\epsilon)$ be such that
$\norma{v_\epsilon}_{H^3(\Omega_\epsilon)} \leq M$ for all $\epsilon>0$. Assume that up to a subsequence $v_\eps|_{\Omega} \rightharpoonup v$ in $H^3(\Omega)$. Let $\varphi$ be a fixed function in $V(\Omega)$. Let $\hat{v}\in L^2(W,
w^{3,2}_{\textup{Per}_Y}(Y\times (-\infty,0)))$ be as in Lemma~\ref{lemma: unfolding convergence poly}. Then
\begin{equation}
\begin{split}
&\int_{Q_\epsilon} D^3 v_{\epsilon} : D^3(T_\epsilon \varphi) \, \diff{x} \rightarrow\\
& - \int_W \int_{Y \times (-1,0)} (D^3_y(\hat{v}) : D^3 (b(\bar{y})(1+y_N)^4) \,\diff{y}\, \frac{\partial\varphi}{\partial x_N}(\bar{x}, 0)
\diff{\bar{x}}.
\end{split}
\end{equation}
\end{lemma}
\begin{proof}
In the following calculations we use the index notation and we drop the summation symbols. We calculate
{\small
\begin{equation}
\label{conti11}
\begin{split}
&\int_{Q_\epsilon} D^3 v_{\epsilon} : D^3(T_\epsilon \varphi) \, \diff{x} = \int_{Q_\epsilon} \frac{\partial^3 v_\epsilon}{\partial x_i
\partial x_j \partial x_h} \frac{\partial^3 (\varphi \circ \Phi_\epsilon)}{\partial x_i \partial x_j \partial x_h} \, \diff{x}\\
&= \int_{Q_\epsilon} \frac{\partial^3 v_\epsilon}{\partial x_i \partial x_j \partial x_h} \frac{\partial^3 \varphi}{\partial x_k \partial x_l
\partial x_m}(\Phi_\epsilon(x))\, \frac{\partial \Phi_\epsilon^{(k)}}{\partial x_i} \frac{\partial \Phi_\epsilon^{(l)}}{\partial x_j}
\frac{\partial \Phi_\epsilon^{(m)}}{\partial x_h} \, \diff{x}\\
&+ \int_{Q_\epsilon} \frac{\partial^3 v_\epsilon}{\partial x_i \partial x_j \partial x_h} \frac{\partial^2 \varphi}{\partial x_k \partial
x_l}(\Phi_\epsilon(x))\, \Big[\frac{\partial \Phi_\epsilon^{(k)}}{\partial x_i} \frac{\partial^2 \Phi_\epsilon^{(l)}}{\partial x_j \partial
x_h} + \frac{\partial \Phi_\epsilon^{(k)}}{\partial x_j} \frac{\partial^2 \Phi_\epsilon^{(l)}}{\partial x_i \partial x_h} + \frac{\partial
\Phi_\epsilon^{(k)}}{\partial x_h} \frac{\partial^2 \Phi_\epsilon^{(l)}}{\partial x_i \partial x_j}\Big] \, \diff{x},\\
&+ \int_{Q_\epsilon} \frac{\partial^3 v_\epsilon}{\partial x_i \partial x_j \partial x_h} \frac{\partial \varphi}{\partial
x_k}(\Phi_\epsilon(x))\, \frac{\partial^3 \Phi_\epsilon^{(k)}}{\partial x_i \partial x_j \partial x_h}\, \diff{x}.
\end{split}
\end{equation}}
It is not difficult to prove that the first integral in the right-hand side of~\eqref{conti11} vanishes as $\epsilon \to 0$, see the proof of
\cite[Proposition 2]{FerLamb}.
We then consider the second integral in the right hand side of~\eqref{conti11}. Note that all the terms with $l\neq N$ vanish. Thus, without
loss of generality we set $l=N$. Consider separately the case $k \neq N$, and $k = N$.\\
\underline{Case $k \neq N$}: by the exact integration formula \eqref{eq: exact int formula} we obtain
\[
\begin{split}
\Bigg\lvert\int_{Q_\epsilon} \frac{\partial^3 v_\epsilon}{\partial x_i \partial x_j \partial x_h} &\frac{\partial^2 \varphi}{\partial x_k
\partial x_N}(\Phi_\epsilon(x))\, \delta_{ki} \frac{\partial^2 \Phi_\epsilon^{(N)}}{\partial x_j \partial x_h } \, \diff{x}\Bigg\rvert \\
&\leq C \epsilon^{1/2}\norma{\eps^{-5/2} \hat{v}_\epsilon}_{W^{3,2}(\widehat{W}_\eps \times Y \times (-1,0))} \norma*{\frac{\partial
\varphi}{\partial x_k \partial x_N}}_{L^2(Q_\eps)} \to 0,
\end{split}
\]
as $\eps \to 0$. \\
\underline{Case $k=N$}: in this case \eqref{eq: exact int formula} applied to \eqref{conti11} gives
\begin{equation}
\label{proof: vanish}
\begin{split}
&\int_{Q_\epsilon} D^3 v_{\epsilon} : D^3(T_\epsilon \varphi) \, \diff{x} =
\epsilon^{-5} \int_{\widehat{W_\epsilon}\times Y \times (-1,0)}\frac{\partial^3\hat{v_\epsilon}}{\partial y_i \partial y_j \partial y_h}
\frac{\partial^2 \varphi}{\partial x_N^2}(\widehat{\Phi_\epsilon}(y))\,\cdot\\
&\cdot\Bigg[ \frac{\partial \widehat{\Phi_\epsilon}^{(N)}}{\partial y_i} \frac{\partial^2 \widehat{\Phi_\epsilon}^{(N)}}{\partial y_j \partial
y_h} + \frac{\partial \widehat{\Phi_\epsilon}^{(N)}}{\partial y_j} \frac{\partial^2 \widehat{\Phi_\epsilon}^{(N)}}{\partial y_i \partial y_h}
+ \frac{\partial \widehat{\Phi_\epsilon}^{(N)}}{\partial y_h} \frac{\partial^2 \widehat{\Phi_\epsilon}^{(N)}}{\partial y_i \partial y_j}
\Bigg]\, \diff{\bar{x}} \diff{y},
\end{split}
\end{equation}
and since we are summing on the indexes $i,j,h \in {1,\dots, N}$, \eqref{proof: vanish} equals
 \[ 3 \epsilon^{-5} \int_{\widehat{W_\epsilon}\times Y \times (-1,0)} \frac{\partial^3 \hat{v_\epsilon}}{\partial y_i \partial y_j \partial
 y_h} \frac{\partial^2 (\varphi(\widehat{\Phi_\epsilon}(y)))}{\partial x_N^2}\, \frac{\partial \widehat{\Phi_\epsilon}^{(N)}}{\partial y_i}
 \frac{\partial^2 \widehat{\Phi_\epsilon}^{(N)}}{\partial y_j \partial y_h} \, \diff{\bar{x}} \diff{y}.\]
Note now that
\[
\frac{\partial \widehat{\Phi_\epsilon}^{(k)}}{\partial y_i } =
\begin{cases}
\epsilon\delta_{ki}, &\textup{if $k \neq N$},\\
\epsilon\delta_{Ni} - \epsilon\widehat{\frac{\partial h_\epsilon}{\partial x_i}}, &\textup{if $k=N$}.
\end{cases} \qquad
\frac{\partial^2 \widehat{\Phi_\epsilon}^{(k)}}{\partial y_i \partial y_j } =
\begin{cases}
0, &\textup{if $k \neq N$},\\
-\epsilon^2\widehat{\frac{\partial^2 h_\epsilon}{\partial x_i \partial x_j}}, &\textup{if $k=N$}.
\end{cases}
\]
Thus, we have
\begin{equation}
\label{conti12}
\begin{split}
& 3 \epsilon^{-5} \int_{\widehat{W_\epsilon}\times Y \times (-1,0)} \frac{\partial^3 \hat{v_\epsilon}}{\partial y_i \partial y_j \partial y_h}
\frac{\partial^2 (\varphi(\widehat{\Phi_\epsilon(y)}))}{\partial x_N^2}\, \frac{\partial \widehat{\Phi_\epsilon}^{(N)}}{\partial y_i}
\frac{\partial^2 \widehat{\Phi_\epsilon}^{(N)}}{\partial y_j \partial y_h } \, \diff{\bar{x}} \diff{y}\\
&= - 3\epsilon^{-2} \int_{\widehat{W_\epsilon}\times Y \times (-1,0)} \frac{\partial^3 \hat{v_\epsilon}}{\partial y_i \partial y_j \partial
y_h} \frac{\partial^2 (\varphi(\widehat{\Phi_\epsilon}(y)))}{\partial x_N^2} \Bigg(\delta_{Ni} - \widehat{\frac{\partial h_\epsilon}{\partial
x_i}}\Bigg) \widehat{\frac{\partial^2 h_\epsilon}{\partial x_j\partial x_h}} \, \diff{\bar{x}} \diff{y}.
\end{split}
\end{equation}
It is not difficult to see that the right-hand side of \eqref{conti12} vanishes as $\epsilon \to 0$, due to \eqref{eq: exact int formula} and Lemma \ref{lemma: convergence b weak}.
It remains to treat only the third integral in the right hand side of \eqref{conti11}.  We apply the exact integration formula \eqref{eq:
exact int formula} in order to obtain
\[
\begin{split}
&\epsilon \int_{\widehat{W_\epsilon}\times Y \times (-1,0)} \widehat{\frac{\partial^3 v_\epsilon}{\partial x_i \partial x_j \partial x_h}}
\frac{\partial \varphi}{\partial x_N}(\widehat{\Phi_\epsilon}(y))\, \widehat{\frac{\partial^3 \Phi_\epsilon^{(N)}}{\partial x_i \partial x_j
\partial x_h}} \,\diff{\bar{x}} \diff{y}\\
&= -\int_{\widehat{W_\epsilon}\times Y \times (-1,0)} \Bigg[\epsilon^{-5/2}\frac{\partial^3 \widehat{v_\epsilon}}{\partial y_i \partial y_j
\partial y_h}\Bigg] \: \Bigg[\frac{\partial \varphi}{\partial x_N}(\widehat{\Phi_\epsilon}(y))\Bigg]\:
\Bigg[\epsilon^{1/2}\widehat{\frac{\partial^3 h_\eps}{\partial x_i \partial x_j \partial x_h}}\Bigg] \,\diff{\bar{x}} \diff{y}.\\
\end{split}
\]
By Lemma \ref{lemma: unfolding convergence poly} it is clear that $ \epsilon^{-5/2}\frac{\partial^3 \widehat{v_\epsilon}}{\partial y_i \partial y_j \partial y_h} \to \frac{\partial^3 \hat{v}}{\partial y_i
\partial y_j \partial y_h}$,
weakly in $L^2(W \times Y \times (-\infty,0))$ as $\eps \to 0$. Moreover, by Lemma \ref{lemma: convergence b weak} $\epsilon^{1/2}\widehat{\frac{\partial^3 \Phi_\epsilon^{(N)}}{\partial x_i \partial x_j \partial x_h}} \to -\frac{\partial^3
(b(\bar{y})(1+y_N)^4)}{\partial y_i \partial y_j \partial y_h}$,
uniformly in $W \times Y \times (-1,0)$ as $\eps \to 0$. Hence,
\[
\begin{split}
&\int_{\widehat{W_\epsilon}\times Y \times (-1,0)} \Bigg[\epsilon^{-5/2}\frac{\partial^3 \widehat{v_\epsilon}}{\partial y_i \partial y_j
\partial y_h}\Bigg] \: \Bigg[\frac{\partial \varphi}{\partial x_N}(\widehat{\Phi_\epsilon}(y))\Bigg]\:
\Bigg[\epsilon^{1/2}\widehat{\frac{\partial^3 \Phi_\epsilon^{(N)}}{\partial x_i \partial x_j \partial x_h}}\Bigg] \,\diff{\bar{x}} \diff{y}\\
&\to -\int_{W\times Y \times (-1,0)} \frac{\partial^3 \hat{v}}{\partial y_i \partial y_j \partial y_h} \: \frac{\partial \varphi}{\partial
x_N}(\bar{x},0)\: \frac{\partial^3 (b(\bar{y})(1+y_N)^4)}{\partial y_i \partial y_j \partial y_h} \,\diff{\bar{x}} \diff{y}.\\
\end{split}
\]
as $\eps \to 0$.
\end{proof}

The previous discussion yields the following
\begin{theorem}
\label{thm: macroscopic limit weak}
Let $f_\epsilon \in L^2(\Omega_\epsilon)$ and $f\in L^2(\Omega)$ be such that $f_\epsilon|_{\Omega} \rightharpoonup f$ in $L^2(\Omega)$. Let
$v_\epsilon\in H^3(\Omega_\epsilon) \cap H^1_0(\Omega_\epsilon)$ be the solutions to $A_{\Omega_\epsilon} v_\epsilon = f_\epsilon$.
Then, possibly passing to a subsequence, there exists $v\in H^3(\Omega) \cap H^1_0(\Omega)$ and $\hat{v} \in
L^2(W,w^{3,2}_{Per_Y}(Y\times(\infty,0)))$ such that $v_\epsilon|_{\Omega}\rightharpoonup v$ in $H^3(\Omega)$, $v_\epsilon|_{\Omega} \to v$ in $L^2(\Omega)$
and such that statements $(a)$ and $(b)$ in Lemma~\ref{lemma: unfolding convergence poly} hold. Moreover,
\begin{multline*}
- \int_W \int_{Y \times (-1,0)} (D^3_y(\hat{v}) : D^3 (b(\bar{y})(1+y_N)^4) \,\diff{y}\, \frac{\partial\varphi}{\partial x_N}(\bar{x}, 0)
\diff{\bar{x}}\\
+ \int_{\Omega} D^3 v : D^3\varphi + u\varphi \,\diff{x} = \int_{\Omega} f \varphi\, \diff{x},
\end{multline*}
for all $\varphi \in H^3(\Omega) \cap H^1_0(\Omega)$.
\end{theorem}

\subsection{Microscopic limit.}
Let $\psi \in C^{\infty}(\overline{W} \times \overline{Y} \times ]-\infty, 0])$ be such that $\supp{\psi} \subset C \times \overline{Y} \times
[d,0]$ for some compact set $C\subset W$, $d\in]-\infty, 0[$, and assume that $\psi(\bar{x},\bar{y},0) = 0$ for all $(\bar{x},\bar{y})\in W
\times Y$. Let $\psi$ be $Y$-periodic in the variable $\bar{y}$. We set
\[
\psi_\epsilon(x) = \epsilon^{\frac{5}{2}} \psi \Big(\bar{x},\frac{\bar{x}}{\epsilon}, \frac{x_N}{\epsilon}\Big),
\]
for all $\epsilon >0$, $x \in W \times ]-\infty, 0]$. Then $T_\epsilon \psi_\epsilon \in V(\Omega_\epsilon)$ for sufficiently small
$\epsilon$, hence we can plug it in the weak formulation of the problem in $\Omega_\epsilon$ in order to get
\begin{equation}
\label{eq: tri weak micro 1}
\int_{\Omega_\epsilon} D^3v_\epsilon : D^3(T_\epsilon \psi_\epsilon)\,\diff{x} + \int_{\Omega_\epsilon} v_\epsilon T_\epsilon \psi_\epsilon
\,\diff{x} = \int_{\Omega_\epsilon} f_\epsilon T_\epsilon \psi_\epsilon \,\diff{x}.
\end{equation}
It is not difficult to prove that
\begin{equation}
\label{eq: tri weak micro 2}
\int_{\Omega_\eps} v_\eps T_\eps \psi_\eps\, dx \to 0, \quad\quad \int_{\Omega_\eps} f_\eps T_\eps \psi \eps \, dx \to 0,
\end{equation}
as $\eps \to 0$, and by arguing as in \cite[Eq. (8.20), p. 29]{ArrLamb} we deduce that
\begin{equation}
\label{eq: tri weak micro 3}
\int_{\Omega_\eps \setminus \Omega} D^3 v_\eps : D^3 (T_\eps\psi_\eps) \, dx \to 0,
\end{equation}
as $\eps \to 0$. Moreover, by arguing as in \cite[Lemma 8.47]{ArrLamb} it is possible to prove that
\begin{equation}
\label{eq: tri weak micro 4}
\int_\Omega D^3v_\epsilon : D^3 (T_\epsilon \psi_\epsilon)\,\diff{x} \to \int_{W\times Y\times (-\infty,0)} D_y^3\hat{v}(\bar{x},y) :
D_y^3\psi(\bar{x},y)\,\diff{\bar{x}} \diff{y},
\end{equation}
as $\eps \to 0$. Then we have the following
\begin{theorem}
\label{thm: conditions on v weak}
Let $\hat{v} \in L^2(W,w^{3,2}_{Per_Y}(Y\times(-\infty,0)))$ be the function from Theorem~\ref{thm: macroscopic limit weak}. Then
\[
\int_{W\times Y \times (-\infty,0)} D^3_y\hat{v}(\bar{x},y) : D^3_y\psi (\bar{x},y) \diff{\bar{x}} \diff{y} = 0,
\]
for all $\psi \in L^2(W,w^{3,2}_{Per_Y}(Y\times(-\infty,0)))$ such that $\psi(\bar{x},\bar{y},0) = 0$ on $W\times Y$. Moreover, for any
$i,j=1,\dots, N-1$, we have
\begin{equation}
\label{eq: Casado new}
\frac{\partial^2\hat{v}}{\partial y_i \partial y_j}(\bar{x}, \bar{y},0) = - \frac{\partial^2 b}{\partial y_i \partial y_j}(\bar{y})
\frac{\partial v}{\partial x_N}(\bar{x}, 0) \quad\quad \textup{on $W \times Y$},
\end{equation}
\end{theorem}
\begin{proof}
We need only to prove~\eqref{eq: Casado new} since the first part of the statement follows from \eqref{eq: tri weak micro 1}, \eqref{eq: tri
weak micro 2}, \eqref{eq: tri weak micro 3}, \eqref{eq: tri weak micro 4} (see also the proof of \cite[Theorem 8.53]{ArrLamb}). By applying
Lemma \ref{lemma: degeneration H^3 H^1_0}, case $\alpha = 5/2$ to $v_\eps \in H^3(\Omega_\eps) \cap H^1_0(\Omega_\eps)$ we deduce the validity
of~\eqref{eq: Casado new}.
\end{proof}

\begin{lemma}
\label{lemma: V weak}
There exists $V\in w^{3,2}_{Per_Y}(Y \times (-\infty,0))$ satisfying the equation
\begin{equation}
\label{eq: variational V weak}
\int_{Y \times (-\infty,0)} D^3 V : D^3 \psi \,\diff{y} = 0,
\end{equation}
for all $\psi \in w^{3,2}_{Per_Y}(Y \times (-\infty,0))$ such that $\psi(\bar{y},0) = 0$ on $Y$, and the boundary condition
\[
V(\bar{y},0) = b(\bar{y}), \quad\quad\textup{on $Y$}.
\]
Function $V$ is unique up to a sum of a monomial in $y_N$ of the form $a y_N^2$. Moreover $V \in W^{6,2}_{Per_Y}(Y \times (d,0))$ for any $d < 0$ and it satisfies the equation
\[
\Delta^3 V = 0, \quad\quad\textup{in $Y \times (d,0) $},
\]
subject to the boundary conditions
\[
\begin{cases}
\frac{\partial^2(\Delta V)}{\partial y_N^2} + 2 \frac{\partial^2}{\partial y_N^2}(\Delta_{N-1}V) = 0, &\textup{on $Y $,}\\
\frac{\partial^3 V}{\partial y_N^3} (\bar{y},0) = 0, &\textup{on $Y$}.
\end{cases}
\]
\end{lemma}
\begin{proof}
Existence, uniqueness and regularity of $V$ follows as in \cite[Lemma 8.60]{ArrLamb}. 
Note that in order to find the boundary conditions satisfied by $V$ on $Y$ we need to use the Triharmonic Green Formula \eqref{eq: triharmonic green 2} with $V$ in place of $f$ and $\psi$ in place of $\varphi$. We choose test functions $\psi$ as in the statement with bounded support in the $y_N$-direction. We then deduce that
\[
\begin{split}
\int_{Y \times (-\infty,0)} D^3 V : D^3 \psi \,\diff{y} &= - \int_{Y \times (-\infty,0)} \Delta^3 V  \psi \, \diff{y} + \int_{Y}
\frac{\partial^3 V}{\partial y_N^3} \frac{\partial^2\psi}{\partial y_N^2}\, \diff{\bar{y}}\\
&- \int_{Y} \Bigg(\frac{\partial^2 (\Delta V)}{\partial y_N^2} + 2 \Delta_{N-1}\bigg(\frac{\partial^2 V}{\partial y_N^2} \bigg)\Bigg)
\frac{\partial \psi}{\partial y_N}\, \diff{\bar{y}},
\end{split}
\]
hence $V$ is triharmonic and satisfies the boundary conditions in the statement.
\end{proof}

\begin{theorem}[Characterisation of the strange term]
\label{thm: chara strange term weak}
Let $V$ be the function defined in Lemma~\ref{lemma: V weak}. Let $v,\hat{v}$ be as in Theorem~\ref{thm: macroscopic limit weak}. Then
\[
\hat{v}(\bar{x},y) = - V(y) \frac{\partial v}{\partial x_N}(\bar{x},0) + a(\bar{x})y_N^2.
\]
for some function $a \in L^2(W)$. Moreover we have the following equalities:
\begin{multline}
\label{eq: strange term V}
\int_{Y \times (-\infty,0)} |D^3V|^2 dy = \int_{Y \times (-\infty,0)} D^3 V : D^3(b(\bar{y})(1+y_N^4)) dy\\
= \int_Y \Bigg(\frac{\partial(\Delta^2 V)}{\partial y_N} + \Delta_{N-1}\bigg(\frac{\partial \Delta V}{\partial y_N}\bigg) +
\Delta^2_{N-1}\bigg(\frac{\partial V}{\partial y_N} \bigg) \Bigg) b(\bar{y}) d \bar{y}.
\end{multline}
\end{theorem}
\begin{proof}
Let $\phi$ be the real-valued function defined on $Y \times ]-\infty,0]$ by
\[
\phi(y) =
\begin{cases}
b(\bar{y}) (1+y_N)^4, &\textup{if $-1\leq y_N \leq 0$},\\
0, &\textup{if $ y_N < -1$}.
\end{cases}
\]
Then $\phi \in H^3(Y \times (-\infty,0))$, and $\phi(\bar{y},0) = 0$ for all $\bar{y} \in Y$. Now note that the function $\psi = V - \phi$
is a suitable test-function in equation \eqref{eq: variational V weak}; by plugging it in we get
\[
\int_{Y \times (-\infty,0)} |D^3V|^2\, dy = \int_{Y \times (-\infty,0)} D^3V : D^3(b(\bar{y})(1+y_N)^4)\, dy
\]
By applying \eqref{eq: triharmonic green 2} on the right-hand side of the former equation, and by keeping in
account that $V$ is as in Lemma \ref{lemma: V weak}, so $\Delta^3V = 0$ in $Y \times (d,0)$ for all $d
< 0$, we deduce that
\begin{multline*}
\int_{Y \times (-\infty,0)} D^3V : D^3(b(\bar{y})(1+y_N)^4)\, dy =\\
\int_Y \Bigg(\frac{\partial(\Delta^2 V)}{\partial y_N} + \Delta_{N-1}\bigg(\frac{\partial (\Delta V)}{\partial y_N}\bigg) +
\Delta^2_{N-1}\bigg(\frac{\partial V}{\partial y_N} \bigg) \Bigg) b(\bar{y})\, d\bar{y}.
\end{multline*}
\end{proof}

By Lemma \ref{lemma: V weak} and Theorem \ref{thm: chara strange term weak} it is now easy to deduce (iii) of Theorem \ref{thm: spectral conv
tri weak}.

\begin{proof}[Proof of Theorem \ref{thm: spectral conv tri weak}(iii)]
Note that the function $v$ of Theorem \ref{thm: macroscopic limit weak} satisfies
\begin{multline}
\label{tri weak final eq}
- \int_W \int_{Y \times (-1,0)} (D^3_y(\hat{v}) : D^3 (b(\bar{y})(1+y_N)^4) \,\diff{y}\, \frac{\partial\varphi}{\partial x_N}(\bar{x}, 0)
\diff{\bar{x}}\\
+\int_{\Omega} D^3 v : D^3\varphi + u\varphi \,\diff{x}= \int_{\Omega} f \varphi\, \diff{x},
\end{multline}
for all $\varphi \in H^3(\Omega) \cap H^1_0(\Omega)$. By Theorem \ref{thm: chara strange term weak} the first integral in the left-hand side
of \eqref{tri weak final eq} can be equivalently rewritten as
\[
\int_W \bigg(\int_{Y \times (-\infty,0)} |D^3V|^2 \,\diff{y}\Bigg)\, \frac{\partial v}{\partial x_N}(\bar{x},
0)\frac{\partial\varphi}{\partial x_N}(\bar{x}, 0) \diff{\bar{x}},
\]
where $V$ is defined in Lemma \ref{lemma: V weak}. By \eqref{eq: triharmonic green 2}
{\small
\begin{equation}
\label{proof: tri weak trih green}
\begin{split}
&\int_{\Omega} D^3 v : D^3\varphi \, dx = - \int_{\Omega} \Delta^3 v\, \varphi\, dx + \int_{\partial \Omega} \frac{\partial^3 f}{\partial n^3}
\frac{\partial^2\varphi}{\partial n^2}\, dS\\
&+ \int_{\partial \Omega} \Bigg(  \big( (n^T D^3v)_{\partial \Omega} : D_{\partial\Omega}n \big) - \frac{\partial^2 (\Delta v)}{\partial n^2} - 2 \Div_{\partial \Omega}(D^3v[n\otimes n])_{\partial \Omega} \Bigg) \frac{\partial \varphi}{\partial n}\, dS,
\end{split}
\end{equation}
}for all $\varphi \in H^3(\Omega) \cap H^1_0(\Omega)$. In particular, we deduce that on $W \times \{0\}$ we have the following boundary
integral
{\small
\begin{equation}
\label{proof: tri weak last B I}
\int_W \Bigg(- \frac{\partial^2 (\Delta v)}{\partial x_N^2}(\bar{x}, 0) - 2 \Delta_{N-1}\bigg(\frac{\partial^2 v}{\partial x_N^2}
\bigg)(\bar{x}, 0) +  K_1 \frac{\partial v}{\partial x_N}(\bar{x}, 0)\Bigg)
\frac{\partial\varphi}{\partial x_N}(\bar{x}, 0)\diff{\bar{x}},
\end{equation}
} where $K_1 = \int_{Y \times (-\infty,0)} |D^3V|^2$. Then, by \eqref{tri weak final eq}, \eqref{proof: tri weak trih green}, \eqref{proof: tri weak last B I} and the arbitrariness of $\varphi$ we deduce the statement of Theorem \ref{thm: spectral conv tri weak}, part (iii).
\end{proof}

\section{Appendix (A)}
\label{sec: poly green formula}

We give here a proof of the Triharmonic Green Formula. We refer to Section \ref{sec: auxil} for the tangential calculus notation and related
results. We first note that by using tangential calculus it is possible to prove that
\begin{multline}
\label{eq: decomposition hessian}
D^2f(x) = \bigg(D^2_{\partial \Omega}f(x) +  \frac{\partial}{\partial n}\bigg(\nabla_{\partial \Omega}f(x) \bigg) \otimes n(x) + n(x) \otimes
\nabla_{\partial \Omega}\bigg(\frac{\partial f(x)}{\partial n} \bigg)\\
+ \frac{\partial^2 f(x)}{\partial n^2} n(x) \otimes n(x) \bigg) + \frac{\partial f(x)}{\partial n} D_{\partial \Omega} n(x),
\end{multline}
for all $x \in \partial \Omega$.
Then formula \eqref{eq: decomposition hessian} can be equivalently rewritten as
\begin{multline}
\label{eq: decomposition hessian 2}
D^2f(x) = \bigg(D^2_{\partial \Omega}f(x) +  \nabla_{\partial \Omega}\bigg(\frac{\partial f(x)}{\partial n} \bigg)\otimes n(x) + n(x) \otimes
\nabla_{\partial \Omega}\bigg(\frac{\partial f(x)}{\partial n} \bigg)\\
+ \frac{\partial^2 f(x)}{\partial n^2} n(x) \otimes n(x) \bigg) - (D_{\partial \Omega}n(x))(\nabla_{\partial \Omega}f(x)) \otimes n(x) +
\frac{\partial f(x)}{\partial n} D_{\partial \Omega} n(x),
\end{multline}
for all $x \in \partial \Omega$. Finally, note that if we take the trace on both hand sides of \eqref{eq: decomposition hessian 2} we recover
the classical decomposition formula for the Laplacian at the boundary
\[
\Delta f(x) = \Delta_{\partial \Omega} f(x) + \frac{\partial^2 f(x)}{\partial n^2} + \mathcal{H}(x) \frac{\partial f(x)}{\partial n},
\]
for all $x \in \partial \Omega$, where $\mathcal{H}$ is the curvature of $\partial \Omega$.

\begin{theorem}[Triharmonic Green Formula - general domain]
\label{thm: trih green}
Let $\Omega$ be a bounded domain of $\R^N$ of class $C^{0,1}$. Let $f \in C^6(\overline{\Omega})$, $\varphi \in C^3(\overline{\Omega})$. Then
\begin{multline}
\label{eq: triharmonic green 1}
\int_{\Omega} D^3f : D^3\varphi\, dx  = - \int_{\Omega} \Delta^3 f\, \varphi\, dx + \int_{\partial \Omega} (n^T D^3f) : D^2\varphi\, dS\\
- \int_{\partial \Omega} (n^T D^2(\Delta f))_{\partial \Omega} \cdot \nabla_{\partial \Omega} \varphi\, dS - \int_{\partial \Omega}
\frac{\partial^2(\Delta f)}{\partial n^2} \frac{\partial \varphi}{\partial n}\, dS  + \int_{\partial \Omega} \frac{\partial (\Delta^2
f)}{\partial n} \varphi\, dS.
\end{multline}
If moreover $\Omega$ is of class $C^3$ then
\begin{equation}
\label{eq: triharmonic green 2}
\begin{split}
&\int_{\Omega}D^3f : D^3\varphi\, dx= - \int_{\Omega} \Delta^3 f\, \varphi\, dx + \int_{\partial \Omega} \frac{\partial^3 f}{\partial n^3}
\frac{\partial^2\varphi}{\partial n^2}\, dS\\
&+ \int_{\partial \Omega} \Bigg(  \big( (n^T D^3f)_{\partial \Omega} : D_{\partial\Omega}n \big) - \frac{\partial^2 (\Delta f)}{\partial n^2}
- 2 \Div_{\partial \Omega}(D^3f[n\otimes n])_{\partial \Omega} \Bigg) \frac{\partial \varphi}{\partial n}\, dS\\
&+ \int_{\partial \Omega} \Bigg( \Div^2_{\partial \Omega}\!\big((n^T D^3f)_{\partial \Omega}\big) +  \Div_{\partial \Omega}\big(D_{\partial
\Omega}n(D^3f[n\otimes n])_{\partial \Omega} \big)\\
&\quad\quad\quad\quad\quad\quad\quad\quad\quad\quad\quad\quad+ \frac{\partial (\Delta^2 f)}{\partial n} + \Div_{\partial
\Omega}\big(n^T D^2(\Delta f)\big)_{\partial \Omega}\Bigg) \varphi\, dS.
\end{split}
\end{equation}
\end{theorem}
\begin{proof}
Repeated integrations by parts establish that
{\small \begin{equation}
\label{proof: trih green formula 1}
\begin{split}
&\int_{\Omega} D^3f : D^3 \varphi\, dx= \int_{\Omega} \frac{\partial^3 f}{\partial x_i \partial x_j \partial x_k} \frac{\partial^3
\varphi}{\partial x_i \partial x_j \partial x_k} \, dx\\
&= - \int_{\Omega} \Delta^3 f \varphi\, dx  + \int_{\partial \Omega} (n^T D^3f) : D^2\varphi\, dS - \int_{\partial \Omega} (n^T D^2(\Delta f))
\cdot \nabla \varphi\, dS + \int_{\partial \Omega} \frac{\partial (\Delta^2 f)}{\partial n} \varphi\, dS,
\end{split}
\end{equation}
}where summation symbols on $i,j,k$ from $1$ to $N$ have been dropped. Then \eqref{eq: triharmonic green 1} follows from \eqref{proof: trih
green formula 1} by decomposing the gradient appearing in the third integral on the right-hand side of \eqref{proof: trih green formula 1} in tangential and normal components, see Definition \ref{tangential gradient}.\\
In order to prove \eqref{eq: triharmonic green 2} we need first to decompose the hessian matrix appearing in the first boundary integral on
the right-hand side of \eqref{eq: triharmonic green 1}. By using formula \eqref{eq: decomposition hessian 2} on $D^2\varphi$ we deduce that
\begin{equation}
\label{proof: decomposition 1}
\begin{split}
\int_{\partial \Omega}(n^T D^3f) : D^2\varphi \,dS &=  \int_{\partial \Omega} (n^T D^3f)_{\partial \Omega} : D^2_{\partial \Omega}\varphi\, dS\\
&+ 2 \int_{\partial \Omega} (D^3f[n\otimes n])_{\partial \Omega} \cdot \nabla_{\partial \Omega}\bigg(\frac{\partial \varphi}{\partial n}
\bigg)\, dS\\
&-\int_{\partial \Omega} \Big(D_{\partial \Omega}n(D^3f[n\otimes n])_{\partial \Omega}\Big) \cdot \nabla_{\partial \Omega} \varphi \, dS\\
&+ \int_{\partial \Omega}\big((n^T D^3f)_{\partial \Omega} : D_{\partial \Omega} n \big) \frac{\partial \varphi}{\partial n}\, dS
+\int_{\partial \Omega} \frac{\partial^3 f}{\partial n^3} \frac{\partial^2 \varphi}{\partial n^2}\, dS.
\end{split}
\end{equation}
In \eqref{proof: decomposition 1} the symbol $D^3f[n\otimes n]$ stands for the vector having as $i$-th component $\frac{\partial^3 f}{\partial x_i \partial x_j \partial x_k} n_j n_k$, where sums over $j$ and $k$ are understood. Note also that the third integral on the right-hand side of \eqref{proof: decomposition 1} is deduced from
\[
-\int_{\partial \Omega} (n^T D^3f) : \big(D_{\partial \Omega}n (\nabla_{\partial \Omega} \varphi) \otimes n \big)\, dS,
\]
by using the following equalities
\[
\begin{split}
(n^T D^3f) : \big(D_{\partial \Omega}n (\nabla_{\partial \Omega} \varphi) \otimes n \big) &=  \big(D_{\partial \Omega}n (\nabla_{\partial
\Omega} \varphi)\big)^T (n^T D^3f) n \\
&= (\nabla_{\partial \Omega} \varphi)^T \big((D_{\partial \Omega}n)^T(D^3f[n \otimes n])_{\partial \Omega}\big)\\
&= \big((D_{\partial \Omega}n)(D^3f[n \otimes n])_{\partial \Omega}\big) \cdot \nabla_{\partial \Omega} \varphi.
\end{split}
\]
In the third equality we have used the fact that $D_{\partial \Omega}n$ is a symmetric matrix. Now, since $\Omega$ is of class $C^2$, we plan to apply the Tangential Divergence theorem (see Theorem \ref{thm: tangential div thm}) to the first, the second, and the third integral in the right-hand side of \eqref{proof: decomposition 1}. We consider separately the first integral. Let us note that for every matrix $A =
(a_{ij}(x))_{ij}$ with coefficients $a_{ij} \in C^2(\overline{\Omega})$ and for every function $\psi \in C^2(\overline{\Omega})$, we have
\[
\int_{\partial \Omega} \Div_{\partial \Omega}\!\big((A)_{\partial \Omega}(\nabla_{\partial \Omega} \psi)\big) \, dS = 0
\]
by \eqref{Tangential Divergence Thm}. Here $((A)_{\partial \Omega} )_{ij}= (a_{ij}\circ p)|_{\partial \Omega}$, where $p$ is defined in
Section \ref{sec: auxil}. Hence,
\begin{equation}
\label{proof: second order TGF}
\int_{\partial \Omega} (\Div_{\partial \Omega}(A)_{\partial \Omega}) \cdot \nabla_{\partial \Omega} \psi + (A)_{\partial \Omega} :
D^2_{\partial \Omega} \psi \, dS = 0.
\end{equation}
Finally, a further application of the Tangential Green formula (see \eqref{Tangential Green's formula}) on the first summand on the right-hand side of \eqref{proof: second order TGF} yields
\begin{equation}
\label{second order TGF}
\int_{\partial \Omega} (\Div^2_{\partial \Omega}(A)_{\partial \Omega}) \psi \, dS =  \int_{\partial \Omega} (A)_{\partial \Omega} :
D^2_{\partial \Omega} \psi \, dS
\end{equation}
for all matrix $A \in C^2(\overline{\Omega})^{N\times N}$, for every function $\psi \in C^2(\overline{\Omega})$. Then, by applying Formula
\eqref{second order TGF} to the first integral in the right-hand side of \eqref{proof: decomposition 1} with $A = (n^TD^3f)$ and $\psi = f$,
and by using \eqref{Tangential Green's formula} on the second and third integral in the right-hand side of \eqref{proof: decomposition 1} we
deduce that
\begin{equation}
\label{proof: boundary terms}
\begin{split}
&\int_{\partial \Omega} (n^T D^3f) : D^2\varphi\, dS=  \int_{\partial \Omega} \Div^2_{\partial \Omega}\big((n^T D^3f)_{\partial \Omega}\big)
\varphi\, dS \\
&- 2 \int_{\partial \Omega} \Div_{\partial \Omega}\big((D^3f[n\otimes n])_{\partial \Omega}\big) \frac{\partial \varphi}{\partial
n}\, dS +\int_{\partial \Omega} \Div_{\partial \Omega}\Big(D_{\partial \Omega}n(D^3f[n\otimes n])_{\partial \Omega}\Big) \varphi \, dS \\
&+ \int_{\partial \Omega}\big((n^TD^3f)_{\partial \Omega} : D_{\partial \Omega} n \big) \frac{\partial \varphi}{\partial n}\, dS +\int_{\partial \Omega} \frac{\partial^3 f}{\partial n^3} \frac{\partial^2 \varphi}{\partial n^2}\, dS,
\end{split}
\end{equation}
where we have denoted with $(V)_{\partial \Omega}$ the projection
of $V$ on the tangent plane to $\partial \Omega$, as defined in \S \ref{sec: auxil}.
By applying the Tangential Divergence Theorem to the second boundary integral on the right-hand side of \eqref{eq: triharmonic green 1} we
finally deduce that
\begin{equation}
\label{proof: decomposition gradient}
- \int_{\partial \Omega} (n^T D^2(\Delta f))_{\partial \Omega} \cdot \nabla_{\partial \Omega} \varphi\, dS = \int_{\partial \Omega}
\Div_{\partial \Omega}\!\big(n^T D^2(\Delta f)\big)_{\partial \Omega}\, \varphi\, dS.
\end{equation}
By \eqref{proof: boundary terms} and \eqref{proof: decomposition gradient} we get \eqref{eq: triharmonic green 2}, concluding the proof.
\end{proof}
%

\section{Appendix (B)}
\label{sec: appendix (B)}

\begin{proposition}
\label{prop: degeneration normal derivative}
Let $u \in H^3(\Omega) \cap H^1_0(\Omega)$ be the function defined in the statement of Lemma \eqref{lemma: degeneration H^3 H^1_0}. If $1<
\alpha < 2$ then $\frac{\partial u}{\partial x_N} (\bar{x},0) = 0$ for almost all $\bar{x} \in W$.
\end{proposition}
\begin{proof}
In this proof we use the definition of $\hat{Y}$ and $\hat{u}$ introduced in \eqref{proof: def hat Y tri} and \eqref{proof: def hat u eps
tri}.
Note that
\begin{equation}
\label{proof: decay gradient alfa < 1}
\eps^\alpha \int_W \int_{\hat{Y}} \frac{1}{\eps^2} |\nabla_{\bar{y}} \hat{u}_\eps|^2 + \frac{1}{\eps^{2\alpha}} \bigg\lvert \frac{\partial
\hat{u}_\eps}{\partial y_N} \bigg\rvert^2\, d\bar{x} dy = \int_{\Omega_\eps} |\nabla u_\eps|^2\, dx,
\end{equation}
where we have used formula \eqref{eq: exact int formula}. Since $\alpha < 2$ we deduce that $ \nabla_y \hat{u}_\eps \to 0$ in $L^2(W \times
\hat{Y})^N$. In a similar way one proves that $D_y^\beta \hat{u}_\eps \to 0$ for all the multiindexes $\beta$ such that $1\leq |\beta| \leq
3$. Now note that $\int_W \int_Y |\hat{u}_\eps(\bar{x}, \bar{y},0)|^2 d\bar{y} d\bar{x} = \int_W |u_\eps(\bar{x},0)|^2 d\bar{x} \leq C$,
uniformly in $\eps > 0$. Thus,
\[
\begin{split}
&\int_W \int_{\hat{Y}} |\hat{u}_\eps(\bar{x}, y)|^2 dy d\bar{x}\\
&\leq 2 \int_W \int_{\hat{Y}} |\hat{u}_\eps(\bar{x}, y) - \hat{u}_\eps(\bar{x},\bar{y},0)|^2 dy d\bar{x} + 2(b(\bar{y}) + 1)\int_W \int_Y
|\hat{u}_\eps(\bar{x}, \bar{y},0)|^2 d\bar{y} d\bar{x}\\
&\leq 2\int_W \int_{\hat{Y}} \int_0^{y_N} \left|\frac{\partial \hat{u}_\eps}{\partial y_N}(\bar{x},\bar{y},t) \right|^2 dt |y_N| dy d\bar{x} +
C \leq C',
\end{split}
\]
hence $\hat{u}_\eps$ is uniformly bounded in $L^2(W, H^3(\hat{Y}))$ and up to a subsequence $\hat{u}_\eps \rightharpoonup \hat{u}$ in $L^2(W, H^3(\hat{Y}))$, for some function $\hat{u} \in L^2(W, H^3(\hat{Y}))$. Actually $\hat{u}$ does not depend on $y$; indeed $\nabla_y \hat{u}_\eps \to 0$ in $L^2(W \times \hat{Y})^N$ implies that $\nabla_y \hat{u}= 0$. Since $u_\eps \to u$ weakly in $H^3(\Omega)$, by the Trace Theorem, $u_\eps(\bar{x},0) \to u(\bar{x},0)$ strongly in $L^2(W)$. By Lemma \ref{lemma: easy averages} below, we deduce that
\begin{equation}
\label{proof: converg u_eps bar}
\overline{u_\eps}(\bar{x}) = \frac{1}{\eps^{N-1}} \int_{C_\eps(\bar{x})} u_\eps(\bar{t},0) d\bar{t} \to u(\bar{x},0),
\end{equation}
strongly in $L^2(W)$ as $\eps \to 0$. By a change of variable it is easy to see that $\overline{u_\eps}(\bar{x}) = \int_Y
\hat{u}_\eps(\bar{x}, \bar{z}, 0)\, d\bar{z}$ for almost all $\bar{x} \in W$. By Poincaré inequality it is also easy to prove that
\begin{equation}
\label{proof: conv hat u eps}
\norma*{\hat{u}_\eps - \int_Y \hat{u}_\eps(\cdot, \bar{z}, 0) d\bar{z}}_{L^2(W \times \hat{Y})} \leq C \norma{\nabla_{\bar{y}}\hat{u}_\eps
}_{L^2(W \times \hat{Y})} \to 0,
\end{equation}
as $\eps \to 0$, according to \eqref{proof: decay gradient alfa < 1}. Then, by \eqref{proof: converg u_eps bar} and \eqref{proof: conv hat u
eps} we have
\begin{multline}
\label{proof: conv hat u eps to u}
\norma{\hat{u}_\eps - u(\bar{x},0)}_{L^2(W \times \hat{Y})}\\
\leq \norma*{\hat{u}_\eps - \int_Y \hat{u}_\eps(\cdot, \bar{z}, 0) d\bar{z}}_{L^2(W \times \hat{Y})} + \norma*{ \int_Y \hat{u}_\eps(\cdot,
\bar{z}, 0) d\bar{z} - u(\bar{x},0)}_{L^2(W \times \hat{Y})} \to 0,
\end{multline}
as $\eps \to 0$, which implies that $\hat{u}(\bar{x}) = u(\bar{x},0)$ for almost all $\bar{x} \in W$. Now we unfold the following identity
\[
\begin{split}
&\frac{\partial^2 u_\eps}{\partial x_i \partial x_j}(\bar{x}, g_\eps(\bar{x})) + \frac{\partial^2 u_\eps}{\partial x_i\partial x_N}(\bar{x},
g_\eps(\bar{x})) \frac{\partial g_\eps(\bar{x})}{\partial x_j} + \frac{\partial^2 u_\eps}{\partial x_j\partial x_N}(\bar{x}, g_\eps(\bar{x}))
\frac{\partial g_\eps(\bar{x})}{\partial x_i}\\
&+ \frac{\partial^2 u_\eps}{\partial x^2_N}(\bar{x}, g_\eps(\bar{x})) \frac{\partial g_\eps(\bar{x})}{\partial x_i}\frac{\partial
g_\eps(\bar{x})}{\partial x_j} + \frac{\partial u_\eps}{\partial x_N}(\bar{x}, g_\eps(\bar{x}))\frac{\partial^2 g_\eps(\bar{x})}{\partial x_i
\partial x_j} = 0,
\end{split}
\]
in order to obtain
{\small \begin{equation}
\label{eq: unfolding equality}
\begin{split}
&\frac{1}{\eps^2}\frac{\partial^2 \hat{u}_\eps}{\partial y_i \partial y_j}(\bar{x}, \bar{y},  b(\bar{y})) +
\frac{\eps^{\alpha-1}}{\eps^{\alpha+1}}\frac{\partial^2 \hat{u}_\eps}{\partial y_i\partial y_N}(\bar{x}, \bar{y}, b(\bar{y})) \frac{\partial
b(\bar{y})}{\partial y_j} + \frac{\eps^{\alpha-1}}{\eps^{\alpha + 1}}\frac{\partial^2 \hat{u}_\eps}{\partial y_j\partial y_N}(\bar{x},
\bar{y}, b(\bar{y})) \frac{\partial b(\bar{y})}{\partial y_i}\\
&+ \frac{\eps^{2\alpha-2}}{\eps^{2\alpha}}\frac{\partial^2 \hat{u}_\eps}{\partial y^2_N}(\bar{x}, \bar{y}, b(\bar{y})) \frac{\partial
b(\bar{y})}{\partial y_i}\frac{\partial b(\bar{y})}{\partial y_j} + \frac{\eps^{\alpha-2}}{\eps^{\alpha}} \frac{\partial
\hat{u}_\eps}{\partial y_N}(\bar{x}, \bar{y}, b(\bar{y}))\frac{\partial^2 b(\bar{y})}{\partial y_i \partial y_j} = 0.
\end{split}
\end{equation}}
Note that $\frac{1}{\eps^2}\frac{\partial^2 \hat{u}_\eps}{\partial y_i \partial y_j}(\bar{x}, \bar{y},  b(\bar{y})) \to \frac{\partial^2
u}{\partial x_i \partial x_j}(\bar{x},0) = 0$, and $\frac{1}{\eps^{\alpha+1}}\frac{\partial^2 \hat{u}_\eps}{\partial y_i\partial y_N}(\bar{x},
\bar{y}, b(\bar{y})) \to \frac{\partial^2 u}{\partial x_i\partial x_N}(\bar{x}, 0)$, $\frac{1}{\eps^{2\alpha}}\frac{\partial^2
\hat{u}_\eps}{\partial y^2_N}(\bar{x}, \bar{y}, b(\bar{y})) \to \frac{\partial^2 u}{\partial x^2_N}(\bar{x}, 0)$ as $\eps \to 0$, where the
limits are in $L^2(W \times Y)$. Hence, if $1 < \alpha < 2$ we deduce that all the summands in \eqref{eq: unfolding equality} are vanishing in
$L^2(W \times Y)$ with the possible exception of $ \frac{\eps^{\alpha-2}}{\eps^{\alpha}} \frac{\partial \hat{u}_\eps}{\partial y_N}(\bar{x},
\bar{y}, b(\bar{y}))\frac{\partial^2 b(\bar{y})}{\partial y_i \partial y_j}$. Since equality \eqref{eq: unfolding equality} must hold, this
implies that also this last summand is bounded; hence,
\[
\frac{1}{\eps^{\alpha}} \frac{\partial \hat{u}_\eps}{\partial y_N}(\bar{x}, \bar{y}, b(\bar{y}))\frac{\partial^2 b(\bar{y})}{\partial y_i
\partial y_j} \to 0,
\]
in $L^2(W \times Y)$ as $\eps \to 0$, and consequently $\frac{\partial u}{\partial x_N}(\bar{x}, 0)\frac{\partial^2 b(\bar{y})}{\partial y_i
\partial y_j} = 0$ for almost all $(\bar{x}, \bar{y}) \in (W \times Y)$. Then $\frac{\partial u}{\partial x_N}(\bar{x}, 0) = 0$ for almost all
$\bar{x} \in W$, concluding the proof.

\end{proof}

\begin{lemma}
\label{lemma: easy averages}
Let $(v_\eps)_\eps$ be a sequence of functions in $L^2(\Theta)$, for a given bounded open set $\Theta \subset \R^N$. Let $v \in L^2(\Theta)$,
and assume that $v_\eps \to v$ in $L^2(\Theta)$. For all $\eps > 0$ let $C_\eps(x) = \{y \in \R^N : |x-y| < \eps\}$ and we define
\[
\overline{v_\eps}(x) = \frac{1}{\eps^N}\int_{C_\eps(x)} v_\eps(y)\, dy,
\]
for almost all $x \in \Theta$. Then $\overline{v_\eps} \to v$ in $L^2(\Theta)$ as $\eps \to 0$.
\end{lemma}
\begin{proof}
We claim that
\begin{equation}
\label{claim lebesgue thm}
\overline{v}(x):= \frac{1}{\eps^N}\int_{C_\eps(x)} v(y) dy \to v(x),
\end{equation}
strongly in $L^2(\Theta)$ as $\eps \to 0$. Let $\delta > 0$ be fixed and let $w \in C^1(\Theta) \cap L^2(\Theta)$ such that
$\norma{v-w}_{L^2(\Theta)} \leq \delta$. Then
\[
\begin{split}
&\overline{v}(x) - v(x) = \frac{1}{\eps^N}\int_{C_\eps(x)} (v(y) - v(x))\,dy\\
&= \frac{1}{\eps^N}\int_{C_\eps(x)} (v(y) - w(y))\,dy + (w(x)-v(x)) + \frac{1}{\eps^N}\int_{C_\eps(x)} (w(y) - w(x))\,dy.
\end{split}
\]
Let us define $\Theta^\eps = \{x \in \Theta : \dist(x, \partial \Theta) > \eps\}$. Note that
\[
\begin{split}
\int_{\Theta^\eps} \Bigg\lvert \frac{1}{\eps^N}\int_{C_\eps(x)} (v(y) - w(y) ) dy \Bigg\rvert^2 dx &\leq \int_{\Theta^\eps}
\frac{1}{\eps^N}\int_{C_\eps(x)} |v(y) - w(y)|^2\, dydx \\
&= \int_{\Theta^\eps} |v(y) - w(y)|^2 \Bigg(\frac{1}{\eps^N}\int_{C_\eps(y)}\, dx\Bigg)\,dy \leq C \delta^2
\end{split}
\]
where we have used Jensen's inequality and Tonelli Theorem. Moreover, it is clear that
\[
\norma*{\frac{1}{\eps^N}\int_{C_\eps(x)} (w(y) - w(x))dy}_{L^2(\Theta)} \leq C \eps.
\]
Hence, $ \norma{\overline{v} - v}_{L^2(\Theta^\eps)} \leq C( \delta + \eps ) \leq C' \delta$, concluding the proof of claim \eqref{claim
lebesgue thm}. Now note that
\[
\norma{\overline{v_\eps} - \overline{v}}_{L^2(\Theta^\eps)} \leq \Bigg(\int_{\Theta^\eps}\Bigg( \frac{1}{\eps^N} \int_{C_\eps(x)}|v_\eps(y) -
v(y)|^2 dy\Bigg) dx\Bigg)^{1/2}.
\]
By Tonelli Theorem we can exchange the order of the integrals in order to obtain
\[
\int_{\Theta^\eps}\Bigg( \frac{1}{\eps^N} \int_{C_\eps(x)}|v_\eps(y) - v(y)|^2 dy\Bigg) dx \leq \norma{v_\eps - v}^2_{L^2(\Theta)}
\frac{1}{\eps^N} \int_{C_\eps(y)} dx = \norma{v_\eps - v}^2_{L^2(\Theta)}.
\]
Hence, $\norma{\overline{v_\eps} - \overline{v}}_{L^2(\Theta^\eps)} \leq \norma{v_\eps - v}_{L^2(\Theta)}$; consequently,
\[
\norma{\overline{v_\eps} - v}_{L^2(\Theta^\eps)} \leq \norma{\overline{v_\eps} - \overline{v}}_{L^2(\Theta^\eps)} + \norma{\overline{v} -
v}_{L^2(\Theta^\eps)} \leq \norma{v_\eps - v}_{L^2(\Theta)} + \norma{\overline{v} - v}_{L^2(\Theta^\eps)},
\]
and the right-hand side tends to zero as $\eps \to 0$.
\end{proof}

\section*{Acknowledgments.}
The author is very thankful to Prof. P.D. Lamberti for several discussions and suggestions.\\
The author acknowledge the support of EPSRC, grant EP/T000902/1.

\bibliographystyle{acm}
\bibliography{triharmonic_weak}

\begin{thebibliography}{10}

\bibitem{MR3762319}
{\sc Arrieta, J.~M., Ferraresso, F., and Lamberti, P.~D.}
\newblock Boundary homogenization for a triharmonic intermediate problem.
\newblock {\em Math. Methods Appl. Sci. 41}, 3 (2018), 979--985.

\bibitem{ArrLamb}
{\sc Arrieta, J.~M., and Lamberti, P.~D.}
\newblock Higher order elliptic operators on variable domains. {S}tability
  results and boundary oscillations for intermediate problems.
\newblock {\em J. Differential Equations 263}, 7 (2017), 4222--4266.

\bibitem{ArrVil1}
{\sc Arrieta, J.~M., and Villanueva-Pesqueira, M.}
\newblock Thin domains with doubly oscillatory boundary.
\newblock {\em Math. Methods Appl. Sci. 37}, 2 (2014), 158--166.

\bibitem{ArrVil2}
{\sc Arrieta, J.~M., and Villanueva-Pesqueira, M.}
\newblock Elliptic and parabolic problems in thin domains with doubly weak
  oscillatory boundary.
\newblock {\em Commun. Pure Appl. Anal. 19}, 4 (2020), 1891--1914.

\bibitem{MR1205756}
{\sc Bailey, P.~B., Everitt, W.~N., Weidmann, J., and Zettl, A.}
\newblock Regular approximations of singular {S}turm-{L}iouville problems.
\newblock {\em Results Math. 23}, 1-2 (1993), 3--22.

\bibitem{Boegli2}
{\sc B\"{o}gli, S.}
\newblock Convergence of sequences of linear operators and their spectra.
\newblock {\em Integral Equations Operator Theory 88}, 4 (2017), 559--599.

\bibitem{BuoKen}
{\sc Buoso, D., and Kennedy, J.~B.}
\newblock The bilaplacian with robin boundary conditions.
\newblock {\em ArXiv, 2105.11249\/} (2021).

\bibitem{BuosoLamb}
{\sc Buoso, D., and Lamberti, P.~D.}
\newblock Eigenvalues of polyharmonic operators on variable domains.
\newblock {\em ESAIM Control Optim. Calc. Var. 19}, 4 (2013), 1225--1235.

\bibitem{Burla}
{\sc Burenkov, V.~I., and Lamberti, P.~D.}
\newblock Sharp spectral stability estimates via the {L}ebesgue measure of
  domains for higher order elliptic operators.
\newblock {\em Rev. Mat. Complut. 25}, 2 (2012), 435--457.

\bibitem{CasDiaz}
{\sc Casado-D\'{\i}az, J., Luna-Laynez, M., and Su\'{a}rez-Grau, F.~J.}
\newblock Asymptotic behavior of a viscous fluid with slip boundary conditions
  on a slightly rough wall.
\newblock {\em Math. Models Methods Appl. Sci. 20}, 1 (2010), 121--156.

\bibitem{CasDiazLuLaySuGrau}
{\sc Casado-D\'{\i}az, J., Luna-Laynez, M., and Su\'{a}rez-Grau, F.~J.}
\newblock Asymptotic behavior of the {N}avier-{S}tokes system in a thin domain
  with {N}avier condition on a slightly rough boundary.
\newblock {\em SIAM J. Math. Anal. 45}, 3 (2013), 1641--1674.

\bibitem{CioDamGri2}
{\sc Cioranescu, D., Damlamian, A., and Griso, G.}
\newblock The periodic unfolding method in homogenization.
\newblock {\em SIAM J. Math. Anal. 40}, 4 (2008), 1585--1620.

\bibitem{CioDamGri}
{\sc Cioranescu, D., Damlamian, A., and Griso, G.}
\newblock {\em The periodic unfolding method}, vol.~3 of {\em Series in
  Contemporary Mathematics}.
\newblock Springer, Singapore, 2018.
\newblock Theory and applications to partial differential problems.

\bibitem{CioDo}
{\sc Cioranescu, D., and Donato, P.}
\newblock {\em An introduction to homogenization}, vol.~17 of {\em Oxford
  Lecture Series in Mathematics and its Applications}.
\newblock The Clarendon Press, Oxford University Press, New York, 1999.

\bibitem{ColPro}
{\sc Colbois, B., and Provenzano, L.}
\newblock Eigenvalues of elliptic operators with density.
\newblock {\em Calc. Var. Partial Differential Equations 57}, 2 (2018), Paper
  No. 36, 35.

\bibitem{CoDaDRMus}
{\sc Costabel, M., Dalla~Riva, M., Dauge, M., and Musolino, P.}
\newblock Converging expansions for {L}ipschitz self-similar perforations of a
  plane sector.
\newblock {\em Integral Equations Operator Theory 88}, 3 (2017), 401--449.

\bibitem{CouHil}
{\sc Courant, R., and Hilbert, D.}
\newblock {\em Methods of mathematical physics. {V}ol. {I}}.
\newblock Interscience Publishers, Inc., New York, N.Y., 1953.

\bibitem{DRMus}
{\sc Dalla~Riva, M., and Musolino, P.}
\newblock Moderately close {N}eumann inclusions for the {P}oisson equation.
\newblock {\em Math. Methods Appl. Sci. 41}, 3 (2018), 986--993.

\bibitem{DelZol}
{\sc Delfour, M.~C., and Zol\'{e}sio, J.-P.}
\newblock {\em Shapes and geometries}, second~ed., vol.~22 of {\em Advances in
  Design and Control}.
\newblock Society for Industrial and Applied Mathematics (SIAM), Philadelphia,
  PA, 2011.
\newblock Metrics, analysis, differential calculus, and optimization.

\bibitem{MR110078}
{\sc Federer, H.}
\newblock Curvature measures.
\newblock {\em Trans. Amer. Math. Soc. 93\/} (1959), 418--491.

\bibitem{FerLamb}
{\sc Ferraresso, F., and Lamberti, P.~D.}
\newblock On a {B}abu\v{s}ka paradox for polyharmonic operators: spectral
  stability and boundary homogenization for intermediate problems.
\newblock {\em Integral Equations Operator Theory 91}, 6 (2019), Paper No. 55,
  42.

\bibitem{FerPro}
{\sc Ferraresso, F., and Provenzano, L.}
\newblock On the eigenvalues of the biharmonic operator with neumann boundary
  conditions on a thin set.
\newblock {\em ArXiv, 2108.03969\/} (2021).

\bibitem{FerreroLamb}
{\sc Ferrero, A., and Lamberti, P.~D.}
\newblock Spectral stability for a class of fourth order {S}teklov problems
  under domain perturbations.
\newblock {\em Calc. Var. Partial Differential Equations 58}, 1 (2019), Paper
  No. 33, 57.

\bibitem{HalRau}
{\sc Hale, J.~K., and Raugel, G.}
\newblock Reaction-diffusion equation on thin domains.
\newblock {\em J. Math. Pures Appl. (9) 71}, 1 (1992), 33--95.

\bibitem{Jim}
{\sc Jimbo, S.}
\newblock The singularly perturbed domain and the characterization for the
  eigenfunctions with {N}eumann boundary condition.
\newblock {\em J. Differential Equations 77}, 2 (1989), 322--350.

\bibitem{JimKos}
{\sc Jimbo, S., and Kosugi, S.}
\newblock Spectra of domains with partial degeneration.
\newblock {\em J. Math. Sci. Univ. Tokyo 16}, 3 (2009), 269--414.

\bibitem{Kato}
{\sc Kato, T.}
\newblock {\em Perturbation theory for linear operators}.
\newblock Classics in Mathematics. Springer-Verlag, Berlin, 1995.
\newblock Reprint of the 1980 edition.

\bibitem{MazNazPlaI}
{\sc Maz{'}~ya, V., Nazarov, S., and Plamenevskij, B.}
\newblock {\em Asymptotic theory of elliptic boundary value problems in
  singularly perturbed domains. {V}ol. {I}}, vol.~111 of {\em Operator Theory:
  Advances and Applications}.
\newblock Birkh\"{a}user Verlag, Basel, 2000.
\newblock Translated from the German by Georg Heinig and Christian Posthoff.

\bibitem{MazNazPlaII}
{\sc Maz{'}~ya, V., Nazarov, S., and Plamenevskij, B.}
\newblock {\em Asymptotic theory of elliptic boundary value problems in
  singularly perturbed domains. {V}ol. {II}}, vol.~112 of {\em Operator Theory:
  Advances and Applications}.
\newblock Birkh\"{a}user Verlag, Basel, 2000.
\newblock Translated from the German by Plamenevskij.

\bibitem{MazNaz}
{\sc Maz{'}~ya, V.~G., and Nazarov, S.~A.}
\newblock Paradoxes of the passage to the limit in solutions of boundary value
  problems for the approximation of smooth domains by polygons.
\newblock {\em Izv. Akad. Nauk SSSR Ser. Mat. 50}, 6 (1986), 1156--1177, 1343.

\bibitem{sch}
{\sc Schatzman, M.}
\newblock On the eigenvalues of the {L}aplace operator on a thin set with
  {N}eumann boundary conditions.
\newblock {\em Appl. Anal. 61}, 3-4 (1996), 293--306.

\end{thebibliography}

\end{document}